\documentclass[smallcondensed]{svjour3}      
\smartqed  
\usepackage{graphicx}

\usepackage{amssymb,amsmath, color}
\usepackage{amssymb}
\usepackage{amsbsy}
\usepackage{enumitem,algorithm2e,algorithmic}

\usepackage{ulem}

\usepackage[colorlinks=true]{hyperref}

\newcommand{\R}{{\mathbb R}}

\DeclareMathOperator{\argmin}{argmin}
\DeclareMathOperator{\prox}{prox}

\newcommand{\cH}{{\mathcal H}}

\newcommand{\demi}{\frac{1}{2}}
\newcommand{\ie}{{\it i.e.}\,\,}

\newlength{\textlarg} 
\newcommand{\rinf}{\R\cup\{+\infty\}}
\newcommand{\eqdef}{:=}

\newcommand{\dotp}[2]{\langle #1,\,#2 \rangle}
\newcommand{\norm}[1]{\left\|{#1}\right\|}
\newcommand{\pa}[1]{\left({#1}\right)}

\newcommand{\interior}{{\rm int}\kern 0.06em}
\newcommand{\inte}{{\rm int}\kern 0.06em}
\newcommand{\cl}{{\rm cl}\kern 0.06em}
\newcommand{\zer}{{\rm zer}\kern 0.06em}
\newcommand{\gph}{{\rm gph}\kern 0.06em}
\newcommand{\dom}{{\rm dom}\kern 0.06em}
\newcommand{\pr}{{\rm pr}\kern 0.06em}
\newcommand{\e}{\varepsilon}

\def\d{\delta}

\def\<{\langle}
\def\>{\rangle}

\usepackage{epsfig}
\usepackage{geometry}
 \usepackage{pict2e}

\if
 {
 \usepackage[pageref]{backref}
\renewcommand*{\backrefalt}[4]{%
\ifcase #1 %
(Not cited)%
\or
(Cited on p.~#2)%
\else
(Cited on pp.~#2)%
\fi
}

}
\fi

\usepackage{ulem}

\begin{document}

\title{Accelerated gradient methods with strong convergence to the minimum norm minimizer: a dynamic approach combining time scaling, averaging, and Tikhonov regularization}

\titlerunning{Inertial gradient dynamics with Tikhonov regularization}

\author{Hedy ATTOUCH    \and Zaki CHBANI   \and Hassan RIAHI}

\institute{
Hedy ATTOUCH  \at IMAG, Univ. Montpellier, CNRS, Montpellier, France\\
hedy.attouch@umontpellier.fr,
\and Zaki CHBANI   \and Hassan RIAHI\\
 Cadi Ayyad University \\ S\'emlalia Faculty of Sciences 
 40000 Marrakech, Morocco\\
   chbaniz@uca.ac.ma  \and h-riahi@uca.ac.ma 
}
\maketitle

%\date{October 24, 2022}

\begin{abstract}
  In a Hilbert framework, for convex differentiable optimization, we consider accelerated gradient methods obtained by combining temporal scaling and averaging techniques with Tikhonov regularization.
 We start from the continuous steepest descent dynamic with an additional Tikhonov regularization term whose coefficient vanishes asymptotically.
 We provide an extensive Lyapunov analysis of this first-order evolution equation.
Then we apply to this dynamic the  method of time scaling and averaging recently introduced by Attouch, Bot and Nguyen. We thus obtain an inertial dynamic
which involves viscous damping associated with Nesterov's method, implicit Hessian damping and Tikhonov regularization.
Under an appropriate setting of the parameters, just using Jensen's inequality, without the need for another Lyapunov analysis, we show that the trajectories have at the same time several remarkable properties: they provide a rapid convergence of values,  fast convergence of the  gradients to zero, and
strong convergence to the minimum norm minimizer.
These results complete and improve the previous results obtained by the authors.
\end{abstract}

\medskip


\medskip


%All acknowledgements should be placed in the back of the paper after Conclusions..

%\vspace{2mm}

\section{Introduction}

Throughout the paper, $\mathcal H$ is a real Hilbert space which is endowed with the scalar product $\langle \cdot,\cdot\rangle$, with $\|x\|^2= \langle x,x\rangle    $ for  $x\in \mathcal H$.
Given $f : \mathcal H \rightarrow \mathbb R$  a general convex function, continuously differentiable,
we will develop fast gradient methods with new favorable properties to solve the  minimization problem
\begin{equation}\label{edo0001}
 \min \left\lbrace  f (x) : \ x \in \mathcal H \right\rbrace.
\end{equation}
Our approach is based on  the convergence properties as $t \to +\infty$ of the trajectories $t \mapsto x(t)$  generated by the first-order evolution equation  
\begin{equation}\label{damped-id}
 \dot{x}(t) +  \nabla f (x(t)) + \varepsilon (t) x(t) =0,
\end{equation}
which can be interpreted as a Tikhonov regularization of the continuous  steepest descent dynamics.
As a basic ingredient, (\ref{damped-id}) involves
    a  Tikhonov regularization term with positive coefficient $\varepsilon (\cdot)$ that satisfies $\lim_{t\rightarrow +\infty} \varepsilon(t) = 0$, which preserves the equilibria. 
    
    \noindent Throughout the paper, we assume that the objective function  $f$ and the Tikhonov regularization parameter $\varepsilon (\cdot)$ satisfy the following hypothesis: %\footnote{Then, we will extend our study to the case of a convex lower semicontinuous proper function $f: \cH \to \R \cup \left\lbrace +\infty \right\rbrace$.}:
\begin{align*}
( \mathcal{A}) \;\begin{cases}
(1) \; \; f : \mathcal H \rightarrow \mathbb R \mbox{ is convex, of class } \mathcal C^1,  \nabla f \mbox{ is Lipschitz continuous on  bounded sets}; \vspace{1mm} \\
(2) \; \; S := \mbox{argmin}_{\cH} f \neq \emptyset. \mbox{ We denote by } x^*  \mbox{ the element of minimum norm of } S;   \vspace{1mm}\\
(3)\; \;  \varepsilon : [t_0 , +\infty [ \to \mathbb R^+  \mbox{ is    nonincreasing, of class } \mathcal C^1, \mbox{ such that }\  \lim_{t \to \infty} \varepsilon (t) =0.
\end{cases}
\end{align*}
We first provide an extensive Lyapunov analysis of the first-order evolution equation (\ref{damped-id}).
Then we apply to (\ref{damped-id}) the recent method of time scaling and averaging introduced by Attouch, Bot and Nguyen \cite{ABotNguyen}. Denoting by $t \mapsto \tau(t)$ the time scaling, we thus obtain the system
\vspace{-1mm}
\begin{equation}\label{1st-damped-id5}
\ddot x(t) +   \frac{1+ \ddot\tau(t)}{\dot\tau(t)} \dot x(t) + \nabla f\Big(x(t)+\dot\tau(t)\dot x(t)\Big) + \varepsilon (\tau(t)) \left( x(t)+\dot\tau(t)\dot x(t) \right) =0,
\end{equation}
\vspace{-1mm}
\noindent which involves viscous damping, implicit Hessian damping, and Tikhonov regularization.

\smallskip

\noindent Of particular interest is the case  $\tau(s) = \frac{s^2}{2(\alpha - 1)}$ for some $\alpha>1$. Then  $ \frac{1+ \ddot\tau(s)}{\dot\tau(s)} = \frac{\alpha}{s}$, and  \eqref{1st-damped-id5} becomes
\vspace{-1mm}
\begin{equation}\label{1st-damped-id05}
 \ddot x(s) +\frac{\alpha}{s}  \dot x(s) + \nabla f\left[x(s)+\dfrac{s}{\alpha - 1}\dot x(s)\right] + \varepsilon \left(\dfrac{s^2}{2(\alpha - 1)}\right) \left[x(s)+\dfrac{s}{\alpha - 1}\dot x(s)\right]=0.
\end{equation}
As a remarkable property, a judicious 
choice of $\varepsilon (\cdot)$  gives solution trajectories of (\ref{1st-damped-id05})
that satisfy the  following three favorable properties:

$\bullet$ Fast convergence of the values (with rate $\dfrac{1}{t^2}$), therefore comparable to the  dynamic model introduced by Su, Boyd, and Cand\`es (\cite{SBC}) of the Nesterov accelerated gradient method.

\smallskip

$\bullet$ Fast convergence of the gradients towards zero.

\smallskip

$\bullet$  Strong convergence of the trajectories towards the minimum norm solution. 

\smallskip

\noindent To our knowledge, this is the first time that these three properties are verified simultaneously by the trajectories of the same dynamic. 

For all the systems considered in the article, we take for granted the existence and uniqueness of the solution of the Cauchy problem. For second-order evolution equations, this is obtained by applying the classical Cauchy-Lipschitz theorem to the Hamiltonian formulation (recall that $\nabla f$ is assumed locally Lipschitz continuous). The passage from the local solution to the global solution results from the energy estimations which are established.

At the end of the paper, we will show that many results have  a natural extension to the case of a convex lower semicontinuous proper function
$f:\cH \to \rinf$.

\subsection{Presentation of the results}
We present two model situations corresponding to cases of $\varepsilon (\cdot)$ of particular interest.

\subsubsection{Case $\varepsilon (t)= \dfrac{\delta}{t} $}

The second-order evolution equation
\begin{equation} \label{model-1}
 \ddot x(t) +  \frac{\alpha + 2\delta}{t} \dot x(t) + \nabla f\left[x(t)+\dfrac{t}{\alpha - 1}\dot x(t)\right] + \dfrac{2\delta (\alpha - 1)}{t^2} x(t)=0
\end{equation}
is obtained by taking in (\ref{1st-damped-id05}) $\varepsilon (t)= \dfrac{\delta}{t} $, with  $\delta >1$. 
In Theorem \ref{thm:model-inertial-b}, we show that for  $\alpha >3$ and $\delta >1$, the trajectories generated by  (\ref{model-1})
 have at the same time the following properties:
 
 \smallskip 

\noindent \; $\bullet$ Rapid convergence of values:

$$
f(x(t))-\min_{\cH} f= \mathcal O\left( \frac{1}{t^{2}}  \right) \mbox{ as } \; t \to +\infty .
$$ 	
According to  $\alpha >3$  and $\delta >1$, we have $\alpha + 2\delta >5$. This is in agreement with the fact that taking the damping coefficient $\gamma/t $ with $\gamma$ sufficiently large in the dynamic model of Su, Boyd and Cand\`es (\cite{SBC}) of the accelerated gradient of Nesterov method is beneficial (especially in the presence of strong convexity), see \cite{AC10}.

\smallskip

\noindent \; $\bullet$ Rapid convergence of the gradients towards zero:
$$\| \nabla f(x(t)) \| = \mathcal O\left( \dfrac{1}{t} \right)  \mbox{ as } \; t \to +\infty.$$  

\smallskip

\noindent \; $\bullet$ Strong convergence of  $x(t)$
	 to the minimum norm solution  $x^*$, as  $t \to +\infty $.

\subsubsection{Case $\varepsilon (t)= \dfrac{1}{t^r} $} The second-order evolution equation
\begin{equation}\label{2d-damped-id-p01}
 \ddot x(t) + \left( \frac{\alpha}{t}  +\dfrac{2^r (\alpha - 1)^{r-1}}{t^{2r-1}}  \right)  \dot x(t) + \nabla f\left(x(t)+\dfrac{t}{\alpha - 1}\dot x(t)\right) + 
\dfrac{2^r(\alpha - 1)^r}{t^{2r}} x(t)=0.
\end{equation}
is obtained by taking in (\ref{1st-damped-id05}) $\varepsilon (t)= \dfrac{1}{t^r} $, with  $0<r<1$. 
Under an appropriate setting of the parameters, just using Jensen's inequality, without the need for another Lyapunov analysis, we show 
in Theorem \ref{thm:model-inertial}  that for arbitrary  $0<r<1$, the trajectories generated by 
(\ref{2d-damped-id-p01}) have at the same time the following properties: 

\smallskip

\noindent \; $\bullet$ Rapid convergence of values:
 for $\alpha >1$, and  $0<r<1$
$$
f(x(t))-\min_{\cH} f= \mathcal O\left( \frac{1}{t^{\alpha - 1}} + \frac{1}{t^{2r}} \right) \mbox{ as } \; t \to +\infty.
$$
By taking $\alpha \geq 3$, and $r \simeq 1$, we thus have 
$f(x(t))-\min_{\cH} f= \mathcal O\left( \dfrac{1}{t^{2r}} \right)$. So  
we can approach arbitrarily the optimal convergence rate $1/t^2$.

\smallskip

\noindent \; $\bullet$ Rapid convergence of the gradients towards zero:
$$\| \nabla f(x(t)) \|^2 = \mathcal O\left( \frac{1}{t^{\alpha - 1}} + \frac{1}{t^{2r}} \right)  \mbox{ as } \; t \to +\infty.$$  
 
\noindent \; $\bullet$  Strong convergence of  $x(t)$
	 to the minimum norm solution  $x^*$, as  $t \to +\infty $.

\subsubsection{Comments}

In the above results, the terminology ``implicit Hessian damping" comes from the fact that by a Taylor expansion (as $t \to +\infty$ we have $\dot{x}(t) \to 0$ which justifies using Taylor expansion), we have
\[
\nabla f\left(x(t)+\beta(t)\dot x(t)\right)\approx \nabla f (x(t)) + \beta(t)\nabla^2 f(x(t))\dot{x}(t) ,
\]
hence making the Hessian damping appear indirectly in (\ref{model-1}) and \eqref{2d-damped-id-p01}.
These results improve the previous results obtained by the authors \cite{ACR}, \cite{ABCR}, \cite{ABCR2}, and which were based on Lyapunov's direct analysis of second-order damped inertial dynamics with Tikhonov regularization,
see also \cite{BCL}. In contrast, the Lyapunov analysis is now performed on the initial first-order evolution equation \eqref{damped-id}
which results in a significantly simplified mathematical analysis.
It is also important to note that the scaling and time averaging method allows these questions to be addressed in a unified way.
% For example our results are valid for $\alpha >1$, and so also cover the subcritical case.
The numerical importance of the Hessian driven damping comes from the fact that it allows to significantly attenuates the oscillations which come naturally with inertial systems, see \cite{ACFR,ACFR-Optimisation}, \cite{SDJS}.

\subsection{Historical perspective}

Initially designed for the regularization of ill-posed inverse problems \cite{Tikh,TA}, the  field of application of the Tikhonov regularization was then considerably widened. 
The coupling of first-order in time  gradient systems with a Tikhonov approximation whose coefficient tends asymptotically towards zero has been highlighted in a series of papers  \cite{AlvCab}, \cite{Att2},   \cite{AttCom}, \cite{AttCza2}, \cite{BaiCom}, \cite{Cab}, \cite{CPS}, \cite{Hirstoaga}. 
Our approach builds on several previous works that have paved the way
concerning the coupling of damped second-order in time gradient systems with Tikhonov approximation. First studies  concerned  the heavy ball with friction system of Polyak \cite{Polyak},
where the damping coefficient $\gamma >0$ is  fixed. In   \cite{AttCza1} Attouch and Czarnecki considered the  system
\begin{equation}\label{HBF-Tikh}
 \ddot{x}(t) + \gamma \dot{x}(t) + \nabla f(x(t)) + \varepsilon (t) x(t) =0.
\end{equation}
In the slow parametrization case $\int_0^{+\infty} \varepsilon (t) dt = + \infty$, they proved that  any solution $x(\cdot)$ of \eqref{HBF-Tikh} converges strongly to the minimum norm element of $\argmin f$, see also \cite{Att-Czar-last}, \cite{Cabot-inertiel}, \cite{CEG},  \cite{JM-Tikh}. This hierarchical minimization result contrasts with the case without the Tikhonov regularization term, where the convergence holds only for weak convergence, and the limit depends on the initial data. 

\noindent In the quest for a faster convergence, the following system
with asymptotically vanishing damping
\begin{equation}\label{edo001-0}
 \mbox{(AVD)}_{\alpha, \varepsilon} \quad \quad \ddot{x}(t) + \frac{\alpha}{t} \dot{x}(t) + \nabla f (x(t)) +\varepsilon(t) x(t)=0,
\end{equation}
was studied by Attouch, Chbani, and Riahi in \cite{ACR}.
It is a Tikhonov regularization of the  dynamic
\begin{equation}\label{edo001}
 \mbox{(AVD)}_{\alpha} \quad \quad \ddot{x}(t) + \frac{\alpha}{t} \dot{x}(t) + \nabla f (x(t))=0,
\end{equation}
which was introduced by  Su, Boyd and
Cand\`es in \cite{SBC}. $\mbox{(AVD)}_{\alpha}$ is a low resolution ODE of the   accelerated gradient method of Nesterov \cite{Nest1,Nest2} and  of the Ravine method \cite{AF}, \cite{SBC}.
$ \mbox{(AVD)}_{\alpha}$ has been the subject of many recent studies which have given an in-depth understanding of the Nesterov accelerated gradient method, see  \cite{AAD1}, \cite{ABCR}, \cite{AC10}, \cite{ACPR},\cite{AP}, \cite{CD}, \cite{MME}, \cite{SBC}, \cite{Siegel}, \cite{WRJ}.

Recently, an original  approach has been developed by the authors, which is  based on  the heavy ball with friction method of Polyak in the
\textit{strongly convex case}. In this case, this autonomous dynamic if known to  
provide exponential convergence rates. To take advantage of this very good convergence property, they  considered the nonautonomous dynamic version of  the heavy ball method which at time $t$ is governed by the gradient of the regularized function
$x\mapsto f(x) + \frac{\varepsilon (t)}{2}\|x\|^2$, where the Tikhonov regularization parameter satisfies $\varepsilon (t) \to 0$ as $t\to +\infty$. This idea which was developed in \cite{ABCR}, \cite{ABCR2},  \cite{AL} is detailed below, and will serve us as a comparison to our results. Note that in the following statement the Hessian driven damping is now taken in the explicit form.

\begin{theorem}\label{thm:model-intro}
Take   $0<r<2$, \; $\delta>2$, \; $\beta >0$.
Let $x : [t_0, +\infty[ \to \mathcal{H}$ be a solution trajectory of
		\begin{equation}\label{particular-intro}
		\ddot{x}(t) + \frac{\d}{ \displaystyle{t^{\frac{r}{2}}}}\dot{x}(t) +\beta\nabla^{2} f\left(x(t) \right)\dot{x}(t)+ \nabla f\left(x(t) \right)+ \frac{1}{t^r} x(t)=0.
		\end{equation}
	Then, we have fast convergence of the values, fast convergence of the gradients towards zero, 
	and  strong convergence of the trajectory to the minimum norm 	solution, with the following rates: 
	
	\smallskip
			
$\bullet$ \; $f(x(t))-\min_{\cH} f= \mathcal O \left( \displaystyle\frac{1}{t^{r} }   \right)$  \;  as $ t \to +\infty$ ;
 
$\bullet$ \; $\displaystyle{\int_{t_0}^{+\infty}t^{\frac{3r -2}{2}}\Vert\nabla f(x(t))\Vert^2 dt<+\infty }$;

$\bullet$ \; $ \|x(t) -x_{\varepsilon(t)}\|^2=\mathcal{O}\left(\displaystyle{\dfrac{1}{ t^{\frac{2-r}2}}}\right)$ \;  as $ t \to +\infty$,

where $\varepsilon \mapsto 
x_{\varepsilon} = \argmin_{\xi \in \cH}\left\lbrace f(\xi) + \frac{\varepsilon}{2} \|\xi\|^2 \right\rbrace$ is the viscosity curve, which is known to converge to the minimum norm 	solution.
\end{theorem}
As a major result our approach allows to improve this result by treating the limiting case $r=2$, which corresponds to the viscous damping coefficient of the Su, Boyd and
Cand\`es  dynamic version of the Nesterov accelerated gradient method.

\section{Steepest descent method with vanishing Tikhonov regularization. }
This section is devoted to  
the asymptotic analysis as $t \to +\infty$ of the continuous steepest descent with vanishing Tikhonov regularization
\begin{equation}\label{damped-id-2}
 \dot{x}(t) +  \nabla f (x(t)) + \varepsilon (t) x(t) =0.
\end{equation}
As explained above, our results will be direct consequences of the convergence rates established for this system. We will just apply the time scaling and averaging method to obtain an inertial system with fast convergence properties.

\subsection{Classical results for the continuous steepest descent}

Let us recall some classic facts concerning the asymptotic behavior of the  continuous steepest descent, \ie with $\varepsilon (t)=0$,
see \cite{ABotNguyen}  for a proof of the following theorem. 

\medskip

\begin{theorem}\label{SD_pert_thm}
Suppose that  $f \colon \cH \to \R$ satisfies $(\mathcal A)$.
Let $x \colon \left[ t_{0} , + \infty \right[ \to \cH$ be a solution trajectory of 

\begin{equation}\label{pert SD_0}\tag{SD}	
\dot{x}(t) + \nabla f( x(t)) = 0.
\end{equation}

Then the following properties are satisfied:

\noindent i)  
 (convergence  of  gradients and velocities)\quad	
  $\left\lVert \nabla f \left( x \left( t \right) \right) \right\rVert = o \left( \dfrac{1}{t} \right) \mbox{ and }\,  \left\lVert  \dot{x}(t)\right\rVert = o \left( \dfrac{1}{t} \right)   \mbox{ as  }\,t \to + \infty.$

\noindent ii) (integral estimate of gradients,  velocities)
$
	\displaystyle{\int_{t_0}^{+\infty} }t \left\lVert \nabla f \left( x \left( t \right) \right) \right\rVert ^{2}  dt < +\infty \, \mbox{ and }\, \displaystyle{\int_{t_0}^{+\infty} }t \left\lVert \dot{x}(t) \right\rVert ^{2}  dt < \infty.
$

\noindent iii) (convergence of values)\quad
$
	f \left( x \left( t \right) \right) -\min_{\cH} f = o \left( \dfrac{1}{t} \right) \textrm{ as } t \to + \infty .
$

\noindent  iv) The solution trajectory $x(t)$ converges weakly as $t \to +\infty$, and its limit belongs to $S=\argmin f$.
\end{theorem}
\subsection{Classical results concerning the viscosity curve}

For any $\varepsilon >0$ let us define
\begin{equation}\label{33}
x_{\varepsilon} = \argmin_{\xi \in \cH}\left\lbrace f(\xi) + \frac{\varepsilon}{2} \|\xi\|^2 \right\rbrace.
\end{equation}
Equivalently
\begin{equation}\label{333}
\nabla f(x_{\varepsilon} ) + \varepsilon x_{\varepsilon} =0.
\end{equation}
The mapping $\varepsilon \mapsto x_{\varepsilon}$ is called the viscosity
curve. Its geometric properties will play a key role in our analysis.
Let us recall some of its classical properties:

\begin{lemma}\label{lem-basic-c}{\rm (\cite{Att2})}
Let $x^{*}=\mbox{\rm proj}_{\argmin f} 0$.
We have
 \begin{itemize}
 	\item[$(i)$] \; $\forall \varepsilon >0 \;\; \;  \|x_{\varepsilon}\|\leq \|x^{*}\|$ ;
\label{2a}  \medskip
	\item[$(ii)$] \; $\lim_{\varepsilon \rightarrow 0}\|x_{\varepsilon}-x^{*}\|=0 $. \label{2b}
 \end{itemize}
 \end{lemma}
Let us now recall the differentiability properties of the viscosity curve.
\begin{lemma}\label{lem-basic-cc}{\rm (\cite{Att2}, \cite{AttCom})}
The function $\varepsilon \mapsto x_{\varepsilon}$ is 
 Lipschitz continuous on the compact intervals of $]0, +\infty[$, hence almost everywhere differentiable, and the following inequality holds:
 \begin{equation}\label{44}
\| \frac{d}{d\varepsilon}\left( x_{\varepsilon} \right)\| \leq \frac{\|x^{*}\|}{\varepsilon}.
\end{equation} 
 \end{lemma}
This is a sharp estimation, since, as shown by Torralba  in \cite{Torralba}, there is a convex smooth function $f$ such that the viscosity curve has infinite length.
%This situation is in fact pathological, and we will consider as the regular case the situation where the viscosity curve has a finite length.

\subsection{The steepest descent with vanishing Tikhonov regularization. Preparatory results}
\label{sec:Lyap}

Let us now introduce a vanishing Tikhonov regularization in the continuous steepest descent
\begin{equation}\label{damped-id-b}
 \dot{x}(t) +  \nabla f (x(t)) + \varepsilon (t) x(t) =0.
\end{equation}
We recall that  $\varepsilon : [t_0 , +\infty [ \to \mathbb R^+  $ is a nonincreasing
function of class $\mathcal{C}^{1}$, such that $\lim_{t\rightarrow +\infty} \varepsilon(t) = 0$.
The function
\begin{equation}\label{def:phi}
\varphi_t: \cH \to \mathbb R, \quad
\varphi_t (x) := f(x) + \frac{\varepsilon(t)}{2} \|x\|^2
\end{equation}
will play a central role in the Lyapunov analysis of \eqref{damped-id-b}  via its strong convexity property.
Thus, it is convenient to reformulate \eqref{damped-id-b} with the help of the function $\varphi_t$, which gives, equivalently
\begin{equation}\label{1sans}
\dot{x}(t) + \nabla {\varphi}_{t}(x(t))=0.
\end{equation}
Let us introduce the real-valued function $t \in [t_0, +\infty[ \mapsto E(t) \in \R^+$, a key ingredient of our Lyapunov analysis. It is defined by
\begin{equation}\label{3}
E(t)\eqdef 
\left(\varphi_{t}(x(t))-\varphi_{t}(x_{\varepsilon(t)})\right) +\dfrac{\varepsilon(t)}{2}\|x(t)-x_{\varepsilon(t)}\|^{2}
\end{equation}
where $\varphi_{t}$ has been defined in (\ref{def:phi}),   
Hence $x_{\varepsilon(t)} = \argmin_{\cH}{\varphi}_{t}$.
We need the following lemmas for the proof of the main theorem of this paper.

\begin{lemma}\label{lem-basic-b}
	Let   $x(\cdot): [t_0, + \infty[ \to \cH$ be a solution trajectory of  the continuous steepest descent  with vanishing Tikhonov regularization \eqref{damped-id-b}. Let $t \in [t_0, +\infty[ \mapsto E(t) \in \R^+$ be the energy function defined in \eqref{3}. Then, the following estimates are satisfied:   for any $t\geq t_0$, 
	\begin{eqnarray}
	&&f(x(t))-\min_{\mathcal H}f	
	\leq  E(t)+\dfrac{\varepsilon(t)}{2}\|x^*\|^{2}; \label{keybb-00}\\
	&&\|x(t) - x_{\varepsilon(t)}\|^2  \leq \frac{E(t)}{\varepsilon(t)} \label{est:basic1}.
	\end{eqnarray}
	Therefore, $x(t)$ converges strongly to $x^*$
	as soon as 
	$
	\lim_{t\to +\infty} \displaystyle{\frac{E(t)}{\varepsilon(t)}}=0.
	$
\end{lemma}
 \begin{proof}
 
 $(i)$
 According to the definition of $ \varphi_{t}$,  we have 
 \begin{equation*}
\begin{array}{lll}
f(x(t))-\min_{\mathcal H}f	& = &  \varphi_{t}(x(t))-\varphi_{t}(x^*)+\dfrac{\varepsilon(t)}{2}\left(\|x^*\|^{2}-\|x(t)\|^{2}\right)  \\ 
	& = & \left[\varphi_{t}(x(t))-\varphi_{t}(x_{\varepsilon(t)})\right]+\left[\underbrace{\varphi_{t}(x_{\varepsilon(t)})-\varphi_{t}(x^*)}_{\leq 0}\right]+\dfrac{\varepsilon(t)}{2}\left(\|x^*\|^{2}-\|x(t)\|^{2}\right)\\
	& \leq  &\varphi_{t}(x(t))-\varphi_{t}(x_{\varepsilon(t)})+\dfrac{\varepsilon(t)}{2}\|x^*\|^{2}.
\end{array}
\end{equation*}
By definition of $E(t)$ we have 
 \begin{equation}\label{E_phi}
  \varphi_{t}(x(t))-\varphi_{t}(x_{\varepsilon(t)}) \leq E(t)
  \end{equation}
 which, combined with the above inequality, gives \eqref{keybb-00}.

\smallskip

$(ii)$
By the strong convexity of $\varphi_{t}$, and 
$x_{\varepsilon(t)}:= \argmin_{\cH}\varphi_{t}$, we have
$$
\varphi_{t}(x(t))-\varphi_{t}(x_{\varepsilon(t)}) \geq \frac{\varepsilon (t)}{2}  \|x(t) - x_{\varepsilon(t)}\|^2 .
$$
Returning to the definition of $E(t)$, we get
$$
E(t)-\frac{\varepsilon (t)}{2}  \|x(t) - x_{\varepsilon(t)}\|^2 \geq \frac{\varepsilon (t)}{2}  \|x(t) - x_{\varepsilon(t)}\|^2 ,
$$
which gives \eqref{est:basic1}.\qed
\end{proof}

\begin{lemma}\label{lem1}
The following properties are satisfied:
 \begin{itemize}
 	\item[$(i)$] For each $t \geq t_0$, \; 
 	$\dfrac{d}{dt}\left(\varphi_{t}(x_{\varepsilon(t)})\right)=\frac{1}{2}\dot{\varepsilon}(t)\|x_{\varepsilon(t)}\|^{2}$.
 	
 	\smallskip
 	
 	\item[$(ii)$] The function $t\mapsto x_{\varepsilon(t)}$ is 
 Lipschitz continuous on the compact intervals of $]t_0, +\infty[$, hence almost everywhere differentiable, and the following inequality holds:  for almost every $t \geq t_0$
 	 $$\left\|\dfrac{d}{dt}\left(x_{\varepsilon(t)}\right)\right\|^{2} \leq -\dfrac{\dot{\varepsilon}(t)}{\varepsilon(t)} \left\langle \dfrac{d}{dt}\left(x_{\varepsilon(t)}\right), x_{\varepsilon(t)}\right\rangle.$$
 \end{itemize}
 	Therefore, for almost every $t \geq t_0$
 	$$\left\|\dfrac{d}{dt}\left(x_{\varepsilon(t)}\right)\right\|\leq -\dfrac{\dot{\varepsilon}(t)}{\varepsilon(t)} \|x_{\varepsilon(t)}\|. $$ 
\end{lemma}

\begin{proof} 

 $(i)$  We use the differentiability properties of the Moreau envelope. Recall that, given $\theta >0$
 \begin{eqnarray}
 f_{\theta} (x)&=& \inf_{\xi\in H} \left\lbrace f(\xi)+\frac{1}{2 \theta}\|x-\xi\|^{2}\right\rbrace \label{def:prox-bb0}\\
  &=& f (\prox_{\theta f}(x)) + \frac{1}{2 \theta} \| x - \prox_{\theta f}(x)\| ^2 ,  \label{def:prox-bb}
\end{eqnarray}
  \ie $\prox_{\theta f}(x)$ is the unique point where the infimum in \eqref{def:prox-bb0} is achieved.
  One can consult \cite[section 12.4]{BC} for more details on the Moreau envelope.
 We thus have    
	$$\varphi_{t}(x_{\varepsilon(t)})=\inf_{\xi\in H} \left\lbrace f(\xi)+\frac{\varepsilon(t)}{2}\|\xi-0\|^{2}\right\rbrace=f_{\frac{1}{\varepsilon(t)}}(0).$$ 
	 Since  $ \dfrac{d}{d\theta}f_{\theta}(x)=-\frac{1}{2}\|\nabla f_{\theta}(x)\|^{2}, $ (see  \cite[Appendix, Lemma 3]{ABCR}),    we have:
	$$ \dfrac{d}{dt}f_{\theta(t)}(x)=-\frac{\dot{\theta}(t)}{2}\|\nabla f_{\theta(t)}(x)\|^{2} .$$ 
	Therefore, 
	\begin{equation}\label{4}
	 \dfrac{d}{dt} \varphi_{t}(x_{\varepsilon(t)})=\dfrac{d}{dt}\left(f_{\frac{1}{\varepsilon(t)}}(0)\right)=\frac{1}{2}\dfrac{\dot{\varepsilon}(t)}{\varepsilon^{2}(t)}\|\nabla f_{\frac{1}{\varepsilon(t)}}(0)\|^{2}.
	\end{equation}
On the other hand, we have 
$$ \nabla\varphi_{t}(x_{\varepsilon(t)})=0\Longleftrightarrow \nabla f(x_{\varepsilon(t)})+\varepsilon(t)x_{\varepsilon(t)}=0\Longleftrightarrow x_{\varepsilon(t)}= 
\prox_{ \frac1{\varepsilon (t)} f}(0).
$$	
Since  $\nabla f_{\frac{1}{\varepsilon(t)}}(0)=\varepsilon(t)\left(0-  \prox_{ \frac1{\varepsilon (t)} f}(0)\right), $  we get  
$ \nabla f_{\frac{1}{\varepsilon(t)}}(0)=-\varepsilon(t)x_{\varepsilon(t)} $.  This combined with \eqref{4} gives
$$ \dfrac{d}{dt} \varphi_{t}(x_{\varepsilon(t)})=\frac{1}{2}\dot{\varepsilon}(t)\|x_{\varepsilon(t)}\|^{2}.$$ 
%%%%
\item[$(ii)$] We have 
$$ -\varepsilon(t)x_{\varepsilon(t)}=\nabla f(x_{\varepsilon(t)})\quad\hbox{and}\quad -\varepsilon(t+h)x_{\varepsilon(t+h)}=\nabla f(x_{\varepsilon(t+h)}). $$
According to the monotonicity  of $\nabla f,$ we have
$$ \langle \varepsilon(t)x_{\varepsilon(t)}-\varepsilon(t+h)x_{\varepsilon(t+h)},x_{\varepsilon(t+h)}- x_{\varepsilon(t)}  \rangle \geq 0 ,$$
which implies 
$$ -\varepsilon(t)\|x_{\varepsilon(t+h)}- x_{\varepsilon(t)}\|^{2} + \left(\varepsilon(t)-\varepsilon(t+h)\right) \langle x_{\varepsilon(t+h)},x_{\varepsilon(t+h)}- x_{\varepsilon(t)}  \rangle \geq 0 .$$
After division by $h^{2},$ we obtain 
\begin{equation}\label{eq:diff-Tikh}
 \dfrac{ \left(\varepsilon(t)-\varepsilon(t+h)\right)}{h} \left\langle x_{\varepsilon(t+h)},\dfrac{x_{\varepsilon(t+h)}- x_{\varepsilon(t)}}{h} \right \rangle \geq \varepsilon(t)\left\|\dfrac{x_{\varepsilon(t+h)}- x_{\varepsilon(t)}}{h}\right\|^{2} .
 \end{equation}
We  rely on the differentiability properties of the viscosity curve $\epsilon \mapsto x_{\epsilon}= \argmin \left\lbrace f(\xi) + \frac{\epsilon}{2}\|\xi\|^2 \right\rbrace$.
 According to \cite{AttCom}, \cite{Hirstoaga}, \cite{Torralba}, the viscosity curve is  Lipschitz continuous on the compact intervals of $]0, +\infty[$. So it is absolutely continuous, and  almost everywhere differentiable. Therefore, the mapping $t \mapsto x_{\epsilon (t)}$ satisfies the same differentiability properties. 
 By letting $h \rightarrow 0$ in \eqref{eq:diff-Tikh} we obtain that, for almost every $t\geq t_0$
 $$-\dot{\varepsilon}(t) \left\langle x_{\varepsilon(t)},\dfrac{d}{dt}x_{\varepsilon(t)} \right \rangle \geq \varepsilon(t)\left\|\dfrac{d}{dt}x_{\varepsilon(t)} \right\|^{2} ,$$
which gives the claim. The last statement follows from  Cauchy-Schwarz inequality.\qed
\end{proof}

\subsection{The steepest descent with vanishing Tikhonov regularization: Lyapunov analysis}
In the Lyapunov analysis of the continuous steepest descent with vanishing Tikhonov regularization (\ie $\varepsilon (t) \to 0$ as $t \to +\infty$) that we recall below
\begin{equation}\label{damped-id-b5}
 \dot{x}(t) +  \nabla f (x(t)) + \varepsilon (t) x(t) =0,
\end{equation}
the following function plays a central role 
\begin{equation}\label{def:3bb}
E(t)\eqdef 
\left(\varphi_{t}(x(t))-\varphi_{t}(x_{\varepsilon(t)})\right) +\dfrac{\varepsilon(t)}{2}\|x(t)-x_{\varepsilon(t)}\|^{2},
\end{equation}
together with
\begin{equation}\label{def:gamma}
\gamma(t)\eqdef \exp\left(\displaystyle \int_{t_1}^{t} \varepsilon(s)ds\right).
\end{equation}
Let us state our main convergence result. 
\begin{theorem}\label{strong-conv-thm-b}
	Let  $x(\cdot): [t_0, + \infty[ \to \cH$ be a solution trajectory of the system \eqref{damped-id}.	Define $E(t)$ and $\gamma (\cdot)$ respectively by \eqref{def:3bb} and \eqref{def:gamma}.
	\begin{enumerate}
	\item 	Then,  the following properties are satisfied:  there exists $t_1>0$ such that for all $ t\geq t_1$
\begin{eqnarray}  
&&	E(t)\leq \dfrac{\gamma(t_1)E(t_1)}{\gamma(t)}  -   \dfrac{\|x^{*}\|^{2}}{\gamma(t)}  \displaystyle\int_{t_1}^{t}\dot \varepsilon(s) \gamma(s) ds   ;  \label{Lyap-basic1} \\
&& 	f(x(t))-\min_{\cH} f \leq \dfrac{\gamma(t_1)E(t_1)}{\gamma(t)}  +   \dfrac{\|x^{*}\|^{2}}2\left[\varepsilon(t) -   \dfrac{2}{\gamma(t)}  \displaystyle\int_{t_1}^{t}\dot \varepsilon(s) \gamma(s) ds\right];   \hspace{2cm} \label{contr:fx(t)}\\
&&	 \|x(t) -x_{\varepsilon(t)}\|^2 \leq  \frac{E(t)}{\varepsilon(t)} \leq \dfrac{\gamma(t_1)E(t_1)}{\varepsilon(t)\gamma(t)}   -   \dfrac{\|x^{*}\|^{2}}{\varepsilon(t)\gamma(t)}  \displaystyle\int_{t_1}^{t}\dot \varepsilon(s) \gamma(s) ds.\label{contr:x(t)}
\end{eqnarray}	
	\item Suppose that one of the two following conditions is satisfied:

\medskip
	
	\quad $(i)$\,  $\lim_{t\to +\infty} \frac{E(t)}{\varepsilon(t)}=0$

or

\quad $(ii)$\   $\displaystyle \int_{t_0}^{+\infty} 	\varepsilon(t) dt = + \infty$. 
	 
\smallskip
	 
Then  $x(t)$ converges strongly to $x^*$ as  $t\to +\infty$.
	\end{enumerate}
\end{theorem}
\begin{proof} Since the mapping $t \mapsto x_{\epsilon (t)}$ is absolutely continuous (indeed locally Lipschitz) the classical derivation  chain rule can be applied to compute the derivative of the function $E(\cdot)$, see \cite[section VIII.2]{Bre2}.
According to Lemma \ref{lem1} $i)$,  for almost all $t\geq t_0$
the derivative of  $E(\cdot)$  is given by:
	\begin{equation}\label{6}
	\begin{array}{lll}
	\dot{E}(t) &=&  \langle \nabla\varphi_{t}(x(t)),\dot{x}(t)\rangle+\dfrac{1}{2}\dot{\varepsilon}(t)\|x(t)\|^{2}-\dfrac{1}{2}\dot{\varepsilon}(t)\|x_{\varepsilon(t)}\|^{2}\\
	&&+ \dfrac{1}{2}\dot{\varepsilon}(t) \norm{x(t)-x_{\varepsilon(t)}}^2 + \varepsilon(t)	\langle\frac{d}{dt}(x(t)-x_{\varepsilon(t)})\;,\;x(t)-x_{\varepsilon(t)}  \rangle.
	\end{array}      
	\end{equation}
According to \eqref{1sans} and the $\varepsilon (t)$ strong  convexity of $\varphi_{t}$, we get
\begin{eqnarray}
	\dot{E}(t) &=& -\norm{\nabla\varphi_{t}(x(t))}^2 + \dfrac{1}{2}\dot{\varepsilon}(t)\left(\|x(t)\|^{2} 
	-\norm{x_{\varepsilon(t)}}^{2}\right) + \dfrac{1}{2}\dot{\varepsilon}(t) \norm{x(t)-x_{\varepsilon(t)}}^2\nonumber
	\\
	& &
	- \varepsilon(t)\left\langle \nabla\varphi_{t}(x(t)) \;,\;  x(t)-x_{\varepsilon(t)}\right\rangle  - \varepsilon(t)\left\langle\frac{d}{dt}\left(x_{\varepsilon(t)}\right)\;,\; x(t)-x_{\varepsilon(t)}  \right\rangle \nonumber \\
	& \leq &  -\norm{\nabla\varphi_{t}(x(t))}^2 + \dfrac{1}{2}\dot{\varepsilon}(t)\left(\|x(t)\|^{2} 
	-\norm{x_{\varepsilon(t)}}^{2}\right) + \dfrac{1}{2}\dot{\varepsilon}(t) \norm{x(t)-x_{\varepsilon(t)}}^2  \nonumber \\
	& &
	+ \varepsilon(t)\left( \varphi_{t}(x_{\varepsilon(t)})-\varphi_{t}(x(t))-\frac{\varepsilon(t)}{2} \| x(t)-x_{\varepsilon(t)}\|^{2}\right)  - \varepsilon(t)\left\langle\frac{d}{dt}\left(x_{\varepsilon(t)}\right)\;,\; x(t)-x_{\varepsilon(t)}  \right\rangle. \, \, \label{eq:basic_Lyap_1}
\end{eqnarray}
According to  Lemma \ref{lem-basic-c} (i) and Lemma \ref{lem1} (ii), we obtain 
\begin{eqnarray}
- \left\langle\frac{d}{dt}\left(x_{\varepsilon(t)}\right)\;,\; x(t)-x_{\varepsilon(t)}  \right\rangle &\leq& 
\dfrac{b(t)}{2}\norm{\dfrac{d}{dt} x_{\varepsilon(t)}}^{2} +\dfrac{1}{2b(t)}\|x(t)-x_{\varepsilon(t)}\|^{2} \nonumber\\
&\leq& \dfrac{b(t)\dot \varepsilon^2(t)}{2\varepsilon^2(t)}\|x_{\varepsilon(t)}\|^{2} +\dfrac{1}{2b(t)}\|x(t)-x_{\varepsilon(t)}\|^{2},\label{parameter_b}
\end{eqnarray}
where $b(\cdot)$ is a continuous and positive real function to be choosen later.
Combining   \eqref{eq:basic_Lyap_1}  with \eqref{parameter_b}, we deduce that  
\begin{eqnarray*}
	\dot{E}(t) & \leq &  -\norm{\nabla\varphi_{t}(x(t))}^2 + \dfrac{1}{2}\dot{\varepsilon}(t)\|x(t)\|^{2} 
	 - \dfrac{1}{2}{\varepsilon}(t)\left( {\varepsilon}(t) -\frac{\dot{\varepsilon }(t)}{\varepsilon(t)} -\frac1{b(t)}\right)\norm{x(t)-x_{\varepsilon(t)}}^2\\
	& &
+ \varepsilon(t)\left( \varphi_{t}(x_{\varepsilon(t)})-\varphi_{t}(x(t))\right)  + \dfrac{\dot \varepsilon(t)}2\left(\dfrac{b(t)\dot \varepsilon(t)}{\varepsilon(t)}-1\right)\|x_{\varepsilon(t)}\|^{2} .
\end{eqnarray*}
Let us now build the differential inequality satisfied by $E (\cdot)$.
	\begin{eqnarray}
	&&\dot{E}(t)+\varepsilon(t)E(t)  \leq    -\norm{\nabla\varphi_{t}(x(t))}^2 + \dfrac{1}{2}\dot{\varepsilon}(t)\|x(t)\|^{2}  \nonumber\\
	& & \hspace{0.8cm}- \dfrac{1}{2}{\varepsilon}(t)\left(  -\frac{\dot{\varepsilon }(t)}{\varepsilon(t)} -\frac1{b(t)}\right)\norm{x(t)-x_{\varepsilon(t)}}^2   + \dfrac{\dot \varepsilon(t)}2\left(\dfrac{b(t)\dot \varepsilon(t)}{\varepsilon(t)}-1\right)\|x_{\varepsilon(t)}\|^{2}. \label{eq:basic_Lyap_2}
	\end{eqnarray}
	Let us now choose $b(\cdot)$ in order to make equal to zero the coefficient of $\norm{x(t)-x_{\varepsilon(t)}}^2$, that is 
$$
\frac1{b(t)}=  -\frac{\dot{\varepsilon }(t)}{\varepsilon(t)} 
$$	
Replacing $b(\cdot)$ by this expression in  \eqref{eq:basic_Lyap_2}, we obtain 		
	\begin{eqnarray}
	&&\dot{E}(t)+\varepsilon(t)E(t) \leq -\norm{\nabla\varphi_{t}(x(t))}^2 + \dfrac{1}{2}\dot{\varepsilon}(t)\|x(t)\|^{2}   - \dot \varepsilon(t) \|x_{\varepsilon(t)}\|^{2}\nonumber\\
	& & \hspace{3cm} \leq  - \dot \varepsilon(t) \|x_{\varepsilon(t)}\|^{2}.
	\end{eqnarray}
By taking $W(t)\eqdef \gamma(t)E(t) $, we conclude that
\begin{eqnarray*}
	\dot{W}(t) &=& \gamma(t)\left(\dot{E}(t)+\varepsilon(t)E(t)\right) 
	\\
	&\leq&  - \gamma(t) \dot \varepsilon(t)\|x^*\|^{2},
\end{eqnarray*}	
Therefore 
\begin{equation}\label{15}
	\dfrac{d}{dt}\left(\gamma(t)E(t)\right) \leq  - \gamma(t) \dot \varepsilon(t)\|x^*\|^{2}.
\end{equation}
By integrating \eqref{15} on $[t_1 , t]$, and dividing by $\gamma(t),$ we obtain  our first claim \eqref{Lyap-basic1} 
	\begin{equation}\label{Lyap-basic1b}
	E(t)\leq \dfrac{\gamma(t_1)E(t_1)}{\gamma(t)}  -   \dfrac{\|x^{*}\|^{2}}{\gamma(t)}  \displaystyle\int_{t_1}^{t}\dot \varepsilon(s) \gamma(s) ds .
	\end{equation}
	Coming back to Lemma \ref{lem-basic-b}, and according to
	(\ref{Lyap-basic1b}),  we get, for any $t\geq t_1$, 
	\begin{eqnarray*}
	f(x(t))-\min_{\mathcal H}f	
	&\leq &  E(t)+\dfrac{\varepsilon(t)}{2}\|x^*\|^{2} \\
	&\leq &   \dfrac{\gamma(t_1)E(t_1)}{\gamma(t)}  +   \dfrac{\|x^{*}\|^{2}}2\left[\varepsilon(t) -   \dfrac{2}{\gamma(t)}  \displaystyle\int_{t_1}^{t}\dot \varepsilon(s) \gamma(s) ds\right]
	\end{eqnarray*}
Similarly, according to Lemma \ref{lem-basic-b} and 
	(\ref{Lyap-basic1b}),  we get, for any $t\geq t_1$,	
	\begin{eqnarray*}
	 \|x(t) - x_{\varepsilon(t)}\|^2  & \leq &
	\frac{E(t)}{\varepsilon(t)} \\
	&\leq & \dfrac{\gamma(t_1)E(t_1)}{\varepsilon(t)\gamma(t)}   -   \dfrac{\|x^{*}\|^{2}}{\varepsilon(t)\gamma(t)}  \displaystyle{ \int_{t_1}^{t}\dot \varepsilon(s) \gamma(s) ds}.
	\end{eqnarray*}

\noindent 	By Lemma \ref{lem-basic-c} (ii), $\norm{x_{\varepsilon(t)}-x^*}$ converges to zero.
Therefore, $x(t)$ converges strongly to $x^*$ as soon as 
	 $\lim_{t\to +\infty} \frac{E(t)}{\varepsilon(t)}=0$.\\
	 Finally, according to  Cominetti-Peypouquet-Sorin \cite[Theorem 2]{CPS}, we have that $x(t)$ converges strongly to the minimum norm solution $x^*$ as soon as  $\displaystyle \int_{t_0}^{+\infty} 	\varepsilon(t) dt = + \infty$.
	 \qed
		\end{proof}
%		
%%%%%%%%%%%%%%%%%%%%%%%%%%%%%%%%%

%%%%%%%%%%%%%%%%%%%%%%%%%%
\subsection{Case $\e(t)=\displaystyle\frac{\delta}{t} $, $\delta>1$}\label{sec:particular-cases-b}
 The convergence rate of the values and the strong convergence to the minimum norm solution will be obtained by particularizing  Theorem \ref{strong-conv-thm-b} to this situation. 
	\begin{theorem}\label{thm:model-ab}
	Take $\e(t)=\displaystyle\frac{\delta}{t} $ and  $\delta>1$.
	Let $x : [t_0, +\infty[ \to \mathcal{H}$ be a solution trajectory of
	\begin{equation}\label{eqr1b}
	\dot{x}(t) + \nabla f\left(x(t) \right)+ \frac{\delta}{t} x(t)=0.
	\end{equation}		
Then,  we have  	
	\begin{eqnarray}
	&& \bullet \, E(t) =  \mathcal O \left( \displaystyle\frac{1}{t}   \right) \mbox{ as } \; t \to +\infty;\hspace{8cm}\label{Lyap-basic2b}\\
	&& 
	\bullet \,  f(x(t))-\min_{\cH} f= \mathcal O \left( \displaystyle\frac{1}{t }   \right) \mbox{ as } \; t \to +\infty;\label{contr:fx(t)2b}\\
	&& \bullet \,  \|\nabla f (x(t))\|=\mathcal{O}\left(\dfrac{1}{ t^{\frac{1}{2}}} \right) \mbox{ as } \; t \to +\infty.
	\end{eqnarray}
\hspace{1.3cm }$\bullet$  The solution trajectory $x(\cdot)$ converges strongly to the minimum norm solution $x^*$.
\end{theorem}
%%%%%%%%%%%%%%%%%%%%%%%%%%%	
\begin{proof}
a) Take $\e(t)=\displaystyle\frac{\delta}{t } $ with $\delta >1$. We get
$$
\gamma(t)\eqdef \exp\left(\displaystyle \int_{t_1}^{t} \varepsilon(s)ds\right) = \exp \left(\ln \left( \frac{t}{t_1} \right)^\delta\right) = \left( \frac{t}{t_1} \right)^\delta.
$$
By Theorem \ref{strong-conv-thm-b}, we get, for $t$ large enough
\begin{eqnarray*}  
E(t) &\leq &\dfrac{\gamma(t_1)E(t_1)}{\gamma(t)}  -   \dfrac{\|x^{*}\|^{2}}{\gamma(t)}  \displaystyle\int_{t_1}^{t}\dot \varepsilon(s) \gamma(s) ds   \\
&\leq & \dfrac{C}{t^\delta}  +   \|x^{*}\|^{2}\left( \frac{t_1} {t}\right)^\delta  \displaystyle\int_{t_1}^{t} \frac{\delta}{s^{2}}\left( \frac{s}{t_1} \right)^\delta  ds \\
&\leq & \dfrac{C}{t^\delta}  + \delta \|x^{*}\|^{2}  \frac{1} {t^\delta } \displaystyle\int_{t_1}^{t} s^{\delta -2}  ds  \\
&\leq & \dfrac{C}{t^\delta}  + \frac{\delta}{\delta -1} \|x^{*}\|^{2}  \frac{1} {t }.      
\end{eqnarray*}	
Since $\delta >1$ we deduce that 
\begin{center}
$
E(t) =  \mathcal O \left( \displaystyle\frac{1}{t}   \right) \mbox{ as } \; t \to +\infty.
$
\end{center}
According to Theorem \ref{strong-conv-thm-b} we get
\begin{eqnarray*}  
	f(x(t))-\min_{\cH} f &\leq & \dfrac{\gamma(t_1)E(t_1)}{\gamma(t)}  +   \dfrac{\|x^{*}\|^{2}}2\left[\varepsilon(t) -   \dfrac{2}{\gamma(t)}  \displaystyle\int_{t_1}^{t}\dot \varepsilon(s) \gamma(s) ds\right] \\
 &=&	 \dfrac{C}{t^\delta}
 +  \delta  \dfrac{\|x^{*}\|^{2}}{2t}   +  + \frac{\delta}{\delta -1} \|x^{*}\|^{2}  \frac{1} {t }  .
\end{eqnarray*}	
Since $\delta >1$ we deduce that
\begin{center}
$
f(x(t))-\min_{\cH} f =  \mathcal O \left( \displaystyle\frac{1}{t}   \right) \mbox{ as } \; t \to +\infty.
$
\end{center}
According to the above estimate and Lemma \ref{ext_descent_lemma} we immediately obtain the following  convergence rate of the gradients towards zero
$$ \|\nabla f (x(t))\|=\mathcal{O}\left(\dfrac{1}{ t^{\frac{1}{2}}} \right) \mbox{ as } \; t \to +\infty.
$$
Finally, since  $\displaystyle \int_{t_0}^{+\infty} 	\varepsilon(t) dt = + \infty$, according to  Cominetti-Peypouquet-Sorin \cite[Theorem 2]{CPS}, we have that $x(t)$ converges strongly to the minimum norm solution $x^*$.
\qed
\end{proof}

\subsection{Case $\e(t)=\dfrac{1}{t^{r} } $,  $0<r<1$}\label{sec:particular-cases}
	%
%	\subsection{$\e(t)=\frac{1}{t^{r} } $}\label{sec:particular-case1}
Take $\e(t)=\displaystyle\frac{1}{t^{r} } $,  $0<r<1$, $t_0>0$, and consider the system \eqref{damped-id}. The convergence rate of the values and the strong convergence to the minimum norm solution will be obtained by particularizing  Theorem \ref{strong-conv-thm-b} to this situation. 
	\begin{theorem}\label{thm:model-a}
	Take $\e(t)=\displaystyle\frac{1}{t^{r} } $ and  $0<r< 1$.
	Let $x : [t_0, +\infty[ \to \mathcal{H}$ be a solution trajectory of
	\begin{equation}\label{eqr1}
	\dot{x}(t) + \nabla f\left(x(t) \right)+ \frac{1}{t^r} x(t)=0.
	\end{equation}		
Then,  we have  	
	\begin{eqnarray}
	&&\bullet \,  E(t) =  \mathcal O \left( \displaystyle\frac{1}{t}   \right) \mbox{ as } \; t \to +\infty;\hspace{8cm}\label{Lyap-basic2}\\
	&& \bullet \, 
	f(x(t))-\min_{\cH} f= \mathcal O \left( \displaystyle\frac{1}{t^{r} }   \right) \mbox{ as } \; t \to +\infty;\label{contr:fx(t)2bb}\\
	&& \bullet \,  \|x(t) -x_{\varepsilon(t)}\|^2=\mathcal{O}\left(\dfrac{1}{ t^{1-r}}\right) \mbox{ as } \; t \to +\infty. \label{contr:x(t)2b}\\
	&& \bullet \,  \|\nabla f (x(t))\|=\mathcal{O}\left(\dfrac{1}{ t^{\frac{r}{2}}} \right) \mbox{ as } \; t \to +\infty. \label{contr:nabla f 2b}\\
	&& \bullet \mbox{ The solution trajectory }  x(\cdot) \mbox{ converges strongly to the minimum norm solution } x^*.\label{contr:traj  2b}
	\end{eqnarray}
\end{theorem}
%%%%%%%%%%%%%%%%%%%%%%%%%%%	
\begin{proof}
a) Take $\e(t)=\displaystyle\frac{1}{t^{r} } $ with $r <1$. We get
$$
\gamma(t)\eqdef \exp\left(\displaystyle \int_{t_1}^{t} \varepsilon(s)ds\right) = \exp \left(\left[\frac{t^{1-r}}{1-r} - \frac{t_1^{1-r}}{1-r}\right]\right) = K_0\exp \left(\frac{1}{1-r} t^{1-r}\right)
$$
By Theorem \ref{strong-conv-thm-b}, we get, for $t$ large enough
\begin{eqnarray*}  
E(t) &\leq &\dfrac{\gamma(t_1)E(t_1)}{\gamma(t)}  -   \dfrac{\|x^{*}\|^{2}}{\gamma(t)}  \displaystyle\int_{t_1}^{t}\dot \varepsilon(s) \gamma(s) ds   \\
&\leq & \dfrac{\gamma(t_1)E(t_1)}{\gamma(t)}  +   \|x^{*}\|^{2}\exp \left(-\frac{1}{1-r} t^{1-r}\right)   \displaystyle\int_{t_1}^{t} \frac{r}{s^{r+1}}\exp \left(\frac{1}{1-r} s^{1-r}\right)   ds  .    
\end{eqnarray*}	
To majorize this last expression, we notice that, given a positive parameter $\rho$
\begin{eqnarray*}  	\frac{d}{ds}\left( \frac{1}{\rho s} \exp \left(\frac{1}{1-r} s^{1-r}\right)   \right)&=& \frac{1}{\rho }\left( \frac{1}{s^{1+r}}- \frac{1}{s^{2}} \right)    \exp \left(\frac{1}{1-r} s^{1-r}\right) 
 \\
& \geq & \frac{r}{s^{r+1}}\exp \left(\frac{1}{1-r} s^{1-r}\right)
\end{eqnarray*}	
as soon as $\rho< \frac{1}{r}$ (for $s$ sufficiently large).
Therefore
\begin{eqnarray*}  
E(t) 
&\leq & \dfrac{\gamma(t_1)E(t_1)}{\gamma(t)}  +   \|x^{*}\|^{2}\exp \left(-\frac{1}{1-r} t^{1-r}\right)   \displaystyle\int_{t_1}^{t} \frac{d}{ds}\left( \frac{1}{\rho s} \exp \left(\frac{1}{1-r} s^{1-r}\right)   \right)   ds \\
&\leq & \dfrac{\gamma(t_1)E(t_1)}{\gamma(t)}  +   \frac{\|x^{*}\|^{2}}{\rho t}.
\end{eqnarray*}	
Since the term $\dfrac{\gamma(t_1)E(t_1)}{\gamma(t)} $ converges to zero exponentially, we  get
\begin{center}
$
E(t) =  \mathcal O \left( \displaystyle\frac{1}{t}   \right) \mbox{ as } \; t \to +\infty;
$
\end{center}
Similarly, according to Theorem \ref{strong-conv-thm-b}
\begin{eqnarray*}  
	f(x(t))-\min_{\cH} f &\leq & \dfrac{\gamma(t_1)E(t_1)}{\gamma(t)}  +   \dfrac{\|x^{*}\|^{2}}2\left[\varepsilon(t) -   \dfrac{2}{\gamma(t)}  \displaystyle\int_{t_1}^{t}\dot \varepsilon(s) \gamma(s) ds\right] \\
 &=&	 \dfrac{\gamma(t_1)E(t_1)}{\gamma(t)}  
 +   \dfrac{\|x^{*}\|^{2}}{2t^r}   +   \frac{\|x^{*}\|^{2}}{\rho t}.
\end{eqnarray*}	
Since the term $\dfrac{\gamma(t_1)E(t_1)}{\gamma(t)} $ converges to zero exponentially, we  get
\begin{center}
$
f(x(t))-\min_{\cH} f =  \mathcal O \left( \displaystyle\dfrac{1}{t^r}   \right) \mbox{ as } \; t \to +\infty;
$
\end{center}
According to the above estimate and Lemma \ref{ext_descent_lemma} we immediately obtain the following  convergence rate of the gradients towards zero
$$ \|\nabla f (x(t))\|=\mathcal{O}\left(\dfrac{1}{ t^{\frac{r}{2}}} \right) \mbox{ as } \; t \to +\infty.
$$
Finally, according to Theorem \ref{strong-conv-thm-b} and $
E(t) =  \mathcal O \left( \frac{1}{t}   \right)$, we get
\begin{eqnarray*}  
 \|x(t) -x_{\varepsilon(t)}\|^2 &\leq& \frac{E(t)}{\varepsilon(t)}\leq  
 \dfrac{C}{t}t^r= \dfrac{C}{t^{1-r}}.
\end{eqnarray*}	
Since $1-r>0$ we conclude that $\norm{x_{\varepsilon(t)}-x(t)}$ converges to zero. According to lemma \ref{lem-basic-c} we have  $\lim_{t\rightarrow +\infty}\|x_{\varepsilon(t)}-x^{*}\|=0$ .
	Therefore $x(t)$ converges strongly to $x^*$. 
	Indeed, this could be obtained directly as a consequence of the  Cominetti-Peypouquet-Sorin \cite[Theorem 2]{CPS}. Our Lyapunov analysis provides indeed a convergence rate.	
This completes the proof. \qed
\end{proof}

\begin{remark}
Let us provide another convergence rate of the gradients towards zero, of  interest when $r \leq \demi$.  Let us observe that 
\begin{eqnarray*}
\|\nabla f (x(t))\| & \leq & \|\nabla f (x(t)) - \nabla f (x_{\varepsilon(t)})\| + \| \nabla f (x_{\varepsilon(t)})\| \\
 &\leq & L\|x(t) - x_{\varepsilon(t)}\| + \varepsilon(t) \| x_{\varepsilon(t)}\|  \\
 & = &  \mathcal{O}\left(\dfrac{1}{ t^{\frac{1-r}{2}}} + \dfrac{1}{t^r}\right). 
\end{eqnarray*}	
\end{remark}

\section{Passing from the first-order  to the second-order differential equation}\label{sect_second-order}
Having in view to develop rapid optimization methods, our study focuses on the second-order time evolution system obtained by applying the "time scaling and averaging" method of Attouch, Bot and Ngyuen \cite{ABotNguyen} to the dynamic \eqref{damped-id}.
In doing so, we hope to take advantage of both the fast convergence properties attached to the inertial gradient methods with vanishing damping coefficients and the strong convexity properties attached to Tikhonov's regularization.

\noindent So, let us make the change of time variable $s=\tau (t)$ in  the damped inertial dynamic 
\begin{equation}\label{damped-idz}
 \dot{z}(s) +  \nabla f (z(s)) + \varepsilon (s) z(s) =0,
\end{equation}
where we take $z$ as a state variable and $s$ as a time variable, which will end up with $x$ and $t$ after time  scaling and averaging.  The time scale $\tau (\cdot)$ is a $\mathcal C^2$  increasing function from $[t_0,+\infty[$ to $[s_0,+\infty[$, which satisfies $\lim_{t \to +\infty}\tau (t) = + \infty$.  Setting $s=\tau (t)$ and  multiplying (\ref{damped-idz}) by $\dot\tau(t)>0$, we obtain
\begin{eqnarray}
\dot\tau(t)\dot{z}(\tau(t)) +  \dot\tau(t)\Big(\nabla f (z(\tau(t))) + \varepsilon (\tau(t)) z(\tau(t))\Big) &=& 0.  \label{change var1}
\end{eqnarray}
Set $v(t):= z(\tau(t))$.
By the derivation chain rule, we have
$
\dot{v} (t)= \dot{\tau}(t) \dot{z}(\tau(t))
$.
So reformulating  (\ref{change var1}) in terms of $v(\cdot)$ and its derivatives, we obtain
\begin{equation}\label{1st-damped-id}
\dot{v} (t) +  \dot\tau(t)\nabla f (v(t)) + \dot\tau(t) \varepsilon (\tau(t)) v(t)=0.
\end{equation}
Let us now consider the averaging process which consists in passing from $v$ to $x$ which is defined by means of the first-order linear differential equation
\begin{equation}\label{1st-damped-idd}
v(t)=x(t) + \dot\tau(t)\dot x(t). 
\end{equation}
By temporal derivation of \eqref{1st-damped-idd} we get
\begin{equation}\label{1st-damped-iddd} 
\dot v(t) = \dot x(t) + \ddot\tau(t)\dot x(t) + \dot\tau(t)\ddot x(t).
\end{equation}
Replacing $v(t)$ and  $\dot v(t)$ in \eqref{1st-damped-id}, we get
\begin{equation}\label{1st-damped-id4}
\dot\tau(t) \ddot x(t) +(1+ \ddot\tau(t))\dot x(t) + \dot\tau(t)\nabla f\left[x(t)+\dot\tau(t)\dot x(t)\right] + \dot\tau(t) \varepsilon (\tau(t)) \left[x(t)+\dot\tau(t)\dot x(t)\right]=0.
\end{equation}
Dividing by $\dot\tau(t) >0$, we finally obtain
\begin{equation}\label{1st-damped-id505}
 \ddot x(t) +\frac{1+ \ddot\tau(t)}{\dot\tau(t)}\dot x(t) + \nabla f\left[x(t)+\dot\tau(t)\dot x(t)\right] + \varepsilon (\tau(t)) \left[x(t)+\dot\tau(t)\dot x(t)\right]=0.
\end{equation}
According to the Su, Boyd and Cand\`es \cite{SBC} model for the Nesterov accelerated gradient method, we consider the case where the viscous damping coefficient in \eqref{1st-damped-id505} satisfies $ \frac{1+ \ddot\tau(t)}{\dot\tau(t)} = \frac{\alpha}{t}$ for some $\alpha>1$. Thus $\tau(t) = \dfrac{t^2}{2(\alpha - 1)}$, and the system \eqref{1st-damped-id505} becomes
\begin{equation}\label{1st-damped-id55}
 \ddot x(t) +\frac{\alpha}{t}  \dot x(t) + \nabla f\left[x(t)+\dfrac{t}{\alpha - 1}\dot x(t)\right] + \varepsilon \left(\dfrac{t^2}{2(\alpha - 1)}\right) \left[x(t)+\dfrac{t}{\alpha - 1}\dot x(t)\right]=0.
\end{equation}
Equivalently
\begin{equation}\label{1st-damped-id555}
 \ddot x(t) + \left(  \frac{\alpha}{t} + \dfrac{t}{\alpha - 1}\varepsilon \left(\dfrac{t^2}{2(\alpha - 1)}\right) \right)  \dot x(t) + \nabla f\left[x(t)+\dfrac{t}{\alpha - 1}\dot x(t)\right] + \varepsilon \left(\dfrac{t^2}{2(\alpha - 1)}\right) x(t)=0.
\end{equation}
We can observe that the time dependent parameter $\varepsilon (\cdot)$ enters both the damping coefficient and the Tikhonov regularization term. 
%Clearly $\varepsilon \left(\frac{t^2}{2(\alpha - 1)}\right)$ is still a vanishing  Tikhonov regularization coefficient.
In parallel to the study of the first-order evolution system studied in the previous section, we consider the two cases of the function $\varepsilon (\cdot)$ which is in our hands, namely $\varepsilon (t) =\frac{ \delta}{t}$ then $\varepsilon (t)=\frac1{t^r}$.
We give two different proofs of convergence results, each relying on a specific technique of independent interest.
\subsection{$\varepsilon (t)=\frac{\delta}{t}$ with $\delta>1$ }

%%%%%%%%%%%%%%%%%%%%%%%%%%%%%%%%%%%

The system \eqref{1st-damped-id55} becomes
\begin{equation}\label{2d-damped-id-p111}
 \ddot x(t) +  \frac{\alpha + 2\delta}{t} \dot x(t) + \nabla f\left[x(t)+\dfrac{t}{\alpha - 1}\dot x(t)\right] + \dfrac{2\delta (\alpha - 1)}{t^2} x(t)=0.
\end{equation}
Let us state the main convergence properties of this system.
\begin{theorem}\label{thm:model-inertial-b}
	Take    $\alpha >3$ and $\delta >1$, which gives $\alpha + 2\delta >5$.
	Let $x : [t_0, +\infty[ \to \mathcal{H}$ be a solution trajectory of
	\eqref{2d-damped-id-p111}.		
Then,  the following properties are satisfied. 	
	\begin{eqnarray}
	&&(i) \,\,  f(x(t))-\min_{\cH} f= \mathcal O\left( \frac{1}{t^{2}}  \right) \mbox{ as } \; t \to +\infty;\label{contr:fx(t)22}\\ 
	&& (ii) \,\, \mbox{ There is   strong convergence of } x(t)
	 \mbox{ to the minimum norm solution } x^*.\nonumber\\
	 &&(iii) \,\,
\| \nabla f(x(t)) \| = \mathcal O\left( \dfrac{1}{t} \right)  \mbox{ as } \; t \to +\infty.
	\end{eqnarray}
\end{theorem}

\begin{proof}
According to Theorem \ref{thm:model-ab}, the rescaled function $v(t):= z(\tau(t))$ satisfies
	\begin{eqnarray}
	&&
	f(v(t))-\min_{\cH} f= \mathcal O\left( \frac{1}{t^{2}}  \right)  \mbox{ as } \; t \to +\infty;\label{contr:fv(t)20}
	\end{eqnarray}
	Our objective  is now to obtain the corresponding convergence rate of the values $f(x(t))-\min_{\cH} f$ as $t \to +\infty$, and the strong convergence of 
 $x(t)$ to $x^*$   the minimum norm element of $S$.
 The following proof is inspired by  Attouch, Bot, Ngyuen \cite{ABotNguyen}. It highlights the averaging interpretation of the passage from $v$ to $x$.

$(i)$  Let us rewrite \eqref{1st-damped-idd}  as
\begin{equation}\label{change var28}
	t \dot{x}(t) + (\alpha -1) x(t) = (\alpha -1) v(t).
\end{equation}

\noindent After multiplication of  \eqref{change var28} by $t^{\alpha -2}$, we get equivalently
\begin{equation}\label{change var29}
	t^{\alpha -1} \dot{x}(t) + (\alpha -1)t^{\alpha -2} x(t) = (\alpha -1)t^{\alpha -2} v(t),
\end{equation}
that is
\begin{equation}\label{change var30}
	\frac{d}{dt} \left( t^{\alpha -1}x(t)\right)  = (\alpha -1)t^{\alpha -2} v(t).
\end{equation}
By integrating \eqref{change var30} from $t_{0}$ to $t$, and according to  $ x(t_{0})=x_0$, we obtain
\begin{eqnarray}
	x(t) &=&  \frac{t_{0}^{\alpha -1}}{ t^{\alpha -1}} v(t_0) + \frac{\alpha -1}{t^{\alpha -1}}\int_{t_{0}}^t \theta^{\alpha -2} v(\theta)d\theta \label{def:x}
\end{eqnarray}
where for simplicity we take $v(s_0)=x_0$.  
Then, observe that $x(t)$ can be simply written as follows
\begin{equation}\label{proba-formulation}
	x(t) =   \int_{t_{0}}^t v(\theta)\,  d\mu_{t} (\theta) 
\end{equation}
where $\mu_t$ is the positive  Radon  measure on $[t_{0}, t]$ defined by
$$
\mu_t = \frac{t_{0}^{\alpha -1}}{ t^{\alpha -1}} \delta_{t_{0}} +  (\alpha -1) \frac{\theta^{\alpha -2}}{t^{\alpha -1}} d\theta .
$$
We are therefore led to examining the convergence rate of  $f\left(\int_{t_{0}}^t v(\theta)  d\mu_{t} (\theta)\right) -\min_{\cH} f$ towards zero as $t \to + \infty$.
We have  that $\mu_t$ is a positive Radon measure on $[t_{0}, t]$  whose total mass is equal to $1$. It is therefore a probability measure, and $\int_{t_{0}}^t v(\theta)\,  d\mu_{t} (\theta)$ is obtained by \textbf{averaging} the trajectory $v(\cdot)$ on $[t_{0},t]$ with respect to  $\mu_t$.
From there, we can deduce fast convergence properties of $x(\cdot)$.
According to the convexity of $f$, and  Jensen's inequality, we obtain that
\begin{eqnarray*}
	f\left(\int_{t_{0}}^t v(\theta)\,  d\mu_{t} (\theta)\right)  -\min_{\cH} f &=& (f -\min_{\cH} f ) \left( \int_{t_{0}}^t v(\theta)  d\mu_t (\theta)\right)\\
	&\leq& \int_{t_{0}}^t  \left( f (v(\theta)) -\min_{\cH} f \right) d\mu_t (\theta)
	\\
	&\leq& L_f \int_{t_{0}}^t  \frac{1}{\theta^2} d\mu_t (\theta) \quad \forall t \geq t_0,
\end{eqnarray*}
where the last  inequality above comes from \eqref{contr:fv(t)20}.
According to the definition of $\mu_t$, it yields
\begin{eqnarray*}
	f\left( \int_{t_{0}}^t v(\theta)\,  d\mu_{t} (\theta)\right)  -\min_{\cH} f 
	&\leq&  \frac{L_f t_{0}^{\alpha -3}}{ t^{\alpha -1}}  +    L_f(\alpha -1) \frac{1}{t^{\alpha -1}}  \int_{t_{0}}^t  \theta^{\alpha -4} d\theta\\
	&\leq&\frac{L_f t_{0}^{\alpha -3}}{ t^{\alpha -1}} + L_f \frac{ \alpha-1}{\alpha-3}\left( \frac{1 }{ t^2} - \frac{t_{0}^{\alpha -3}}{t^{\alpha -1}}\right)
	 \quad \forall t \geq t_0.
\end{eqnarray*}
According to  \eqref{proba-formulation} we get 

\medskip

\noindent $\bullet$ For $1 <\alpha <3$, 
$$f(x(t)) -\min_{\cH} f \leq \frac{C}{t^{\alpha -1}}.$$
$\bullet$ For $\alpha >3$, we have $\dfrac{t_0^{\alpha-3}}{ t^{\alpha -1}} \leq \dfrac{1 }{ t^2}$  for every $t \geq t_0$. We therefore obtain 
$$
 f(x(t))-\min_{\cH} f= \mathcal O\left( \frac{1}{t^{2}}  \right) \mbox{ as } \; t \to +\infty
 $$

$(ii)$ According to Theorem \ref{thm:model-ab}, the rescaled function $v(\cdot )$  strongly converges to the minimum norm solution. From  the interpretation of $x$ as an average of $v$, and  using the fact that convergence implies ergodic convergence, we deduce the strong convergence of $x(t)$ towards the minimum norm solution.
This argument is detailed in the next paragraph, so we omit the details here.

\medskip

$(iii)$ We have obtained in $ii)$ that the trajectory $x(\cdot)$ converges. Therefore, it is bounded. Let $L>0$ be the Lipschitz constant of   $\nabla f$ on a ball that contains the trajectory  $x(\cdot)$.   According to the convergence rate of values  and Lemma \ref{ext_descent_lemma} in the Appendix, we immediately obtain the following  convergence rate of the gradients towards zero
\vspace{-7pt}
\begin{equation*}
	\dfrac{1}{2L} \left\lVert \nabla f(x(t)) \right\rVert ^{2} \leq f(x(t)) -\min_{\cH} f =\mathcal O\left( \frac{1}{t^2} \right).
\end{equation*}
So,
$$ \|\nabla f (x(t))\|=\mathcal{O}\left(\dfrac{1}{t} \right) \mbox{ as } \; t \to +\infty.
$$
This completes the proof. \qed
\end{proof}

\subsection{$\varepsilon (t)=\frac1{t^r}$ for $0<r<1$ }

The system \eqref{1st-damped-id55} becomes
\begin{equation}\label{2d-damped-id-p1}
 \ddot x(t) +  \frac{\alpha}{t}\dot x(t) + \nabla f\left[x(t)+\dfrac{s}{\alpha - 1}\dot x(t)\right] + \left(\dfrac{2(\alpha - 1)}{t^2}\right)^{r} \left[x(t)+\dfrac{t}{\alpha - 1}\dot x(t)\right]=0.
\end{equation}
According to Theorem \ref{thm:model-a}, the rescaled function $v(t):= z(\tau(t))$ satisfies
	\begin{eqnarray}
	&&
	f(v(t))-\min_{\cH} f= \mathcal O \left( \displaystyle\frac{1}{\tau(t)^{r} }   \right) \mbox{ as } \; t \to +\infty;\label{contr:fv(t)2}\\
	&& \|v(t) -x_{\varepsilon(\tau(t))}\|^2=\mathcal{O}\left(\dfrac{1}{ \tau(t)^{1-r}}\right) \mbox{ as } \; t \to +\infty. \label{contr:v(t)2}
	\end{eqnarray}
	Our objective  is now to obtain the corresponding convergence rate of the values $f(x(t))-\min_{\cH} f$ as $t \to +\infty$, and strong convergence of 
 $x(t)$ to $x^*$   the minimum norm element of $S$.
%Let us state the main convergence properties of this system, which is formulated with $t$ as a time variable.

\begin{theorem}\label{thm:model-inertial}
	Take   $0<r< 1$ and $\alpha >1$.
	Let $x : [t_0, +\infty[ \to \mathcal{H}$ be a solution trajectory of
\begin{equation}\label{2d-damped-id-p1-t}
\ddot x(t) + \left( \frac{\alpha}{t}  +\dfrac{2^r (\alpha - 1)^{r-1}}{t^{2r-1}}  \right)  \dot x(t) + \nabla f\left(x(t)+\dfrac{t}{\alpha - 1}\dot x(t)\right) + 
\dfrac{2^r(\alpha - 1)^r}{t^{2r}} x(t)=0.
\end{equation}		
Then,  the following properties are satisfied. 	
	\begin{eqnarray}
	&&(i) \,\,  f(x(t))-\min_{\cH} f= \mathcal O\left( \frac{1}{t^{\alpha - 1}} + \frac{1}{t^{2r}} \right) \mbox{ as } \; t \to +\infty;\label{contr:fx(t)2}\\
	&& (ii)  \, \,\mbox{ There is   strong convergence of } x(t) \mbox{ to the minimum norm solution } x^* ; \\
	&& (iii) \,\, \| \nabla f(x(t)) \|^2 = \mathcal O\left( \frac{1}{t^{\alpha - 1}} + \frac{1}{t^{2r}} \right) \mbox{ as } \; t \to +\infty. \label{contr:x(t)2}
	\end{eqnarray}
\end{theorem}
\begin{proof}

$(i)$ \textit{Convergence rates to zero for the values $f(x(t))-\min_\cH f$}.
By definition of $x(\cdot)$
$$
v(s)=x(s)+\dfrac{s}{\alpha - 1}\dot x(s) =\frac1{(\alpha - 1)s^{\alpha - 2}}\frac{d}{ds}(s^{\alpha - 1}x(s)) .
$$
Equivalently
\begin{equation}\label{eq:Eq-dif1}
\frac{d}{ds}(s^{\alpha - 1}x(s)) = ({\alpha - 1})s^{\alpha - 2}v(s).
\end{equation}
From this it is easy to verify that $x$ can be interpreted as an average of $v$.
After integration from $t$ to $t+h$ of \eqref{eq:Eq-dif1}  and division by $(t+h)^{\alpha - 1}$, we get
\begin{eqnarray*}
x(t+h) &= & \left(\frac{t}{t+h}\right)^{\alpha - 1}x(t)+ \frac{1}{(t+h)^{\alpha - 1}}\int_t^{t+h}({\alpha - 1})s^{\alpha - 2}v(s)ds\\
	  &= & \left(\frac{t}{t+h}\right)^{\alpha - 1}x(t)+ \left(1-\left(\frac{t}{t+h}\right)^{\alpha - 1}\right)\frac{1}{(t+h)^{\alpha - 1}-t^{\alpha - 1}} \int_t^{t+h}({\alpha - 1})s^{\alpha - 2}v(s)ds\\
	  &= & \left(\frac{t}{t+h}\right)^{\alpha - 1}x(t)+ \left(1-\left(\frac{t}{t+h}\right)^{\alpha - 1}\right)\frac{1}{(t+h)^{\alpha - 1}-t^{\alpha - 1}} \int_{t^{\alpha - 1}}^{(t+h)^{\alpha - 1}} v\left(\theta^{1/(\alpha - 1)}\right)d\theta,
\end{eqnarray*}
where the last equality comes from the change of time variable $\theta=s^{\alpha - 1}$.

\noindent According to the convexity of the   function $F=f - \inf_{\cH}f$, and using  Jensen's inequality, we obtain
\begin{small}
\begin{eqnarray*}
F(x(t+h)) & \leq &  \left(\frac{t}{t+h}\right)^{\alpha - 1}F(x(t))+ \left(1-\left(\frac{t}{t+h}\right)^{\alpha - 1}\right)F\left(\frac{1}{(t+h)^{\alpha - 1}-t^{\alpha - 1}} \int_{t^{\alpha - 1}}^{(t+h)^{\alpha - 1}} v\left(\theta^{1/(\alpha - 1)}\right)d\theta \right)\\
	& \leq &  \left(\frac{t}{t+h}\right)^{\alpha - 1}F(x(t))+ \left(1-\left(\frac{t}{t+h}\right)^{\alpha - 1}\right)\frac{1}{(t+h)^{\alpha - 1}-t^{\alpha - 1}}  \int_{t^{\alpha - 1}}^{(t+h)^{\alpha - 1}} F\left( v\left(\theta^{1/(\alpha - 1)}\right)\right)d\theta  .
\end{eqnarray*}
\end{small}
Using again the change of time variable $s=\theta^{1/(\alpha - 1)}$, we get
$$
\int_{t^{\alpha - 1}}^{(t+h)^{\alpha - 1}}  F\left( v\left(\theta^{1/(\alpha - 1)}\right)\right)d\theta   =  \int_t^{t+h}({\alpha - 1})s^{\alpha - 2}F(v(s))ds .
$$
It follows
\begin{small}
\begin{eqnarray*}
F(x(t+h))  &\leq&  \left(\frac{t}{t+h}\right)^{\alpha - 1}F(x(t))+ \left(1-\left(\frac{t}{t+h}\right)^{\alpha - 1}\right)\frac{1}{(t+h)^{\alpha - 1}-t^{\alpha - 1}}  \int_{t}^{t+h} ({\alpha - 1})s^{\alpha - 2}F(v(s))ds .
\end{eqnarray*}
\end{small}
Using  \eqref{contr:fv(t)2}, with  $\tau(s) = \dfrac{s^2}{2(\alpha - 1)}$, we obtain the existence of a constant $C>0$ such that
\begin{small}
\begin{eqnarray*}
F(x(t+h)) &\leq&   \left(\frac{t}{t+h}\right)^{\alpha - 1}F(x(t))+ \left(1-\left(\frac{t}{t+h}\right)^{\alpha - 1}\right)\frac{C(\alpha - 1)}{(t+h)^{\alpha - 1}-t^{\alpha - 1}}\int_t^{t+h} \displaystyle s^{\alpha - 2}\left(\dfrac{2(\alpha - 1)}{s^2}\right)^{r}  ds . 
\end{eqnarray*}
\end{small}
Therefore, multiplying by $\frac1h(t+h)^{\alpha - 1}$ and rearranging the terms of this inequality, we conclude
\begin{eqnarray*}
\frac1h\left( (t+h)^{\alpha - 1}F(x(t+h)) -    t^{\alpha - 1}F(x(t))\right) \leq 2^rC(\alpha - 1)^{r+1}\frac1h  \int_t^{t+h} \displaystyle \dfrac{ds}{s^{2r-\alpha +2}}  ds . 
\end{eqnarray*}
By letting $h \rightarrow 0^+$ in the above inequality, we first get  that $\frac{d}{dt}\Gamma (t) \leq 0,$ where
$$
 \Gamma (t):=t^{\alpha - 1}F(x(t)) - 2^rC(\alpha - 1)^{r+1}\int_{s_0}^{t} \displaystyle \dfrac{ds}{s^{2r-\alpha +2}}  ds.
$$
This  implies $\Gamma (t)$ is nonincreasing. Therefore  for each $0<r<1$ and each $t\geq s_0$, we have $\Gamma (t) \leq \Gamma (t_0)$, which gives
$$
t^{\alpha - 1}F(x(t))  \leq \Gamma (t_0) + 2^rC(\alpha - 1)^{r+1}\int_{s_0}^{t} \displaystyle \dfrac{ds}{s^{2r-\alpha +2}}  ds.
$$ 
Consequently,
\begin{equation}\label{ineq-estim}
f (x(t))- \inf_{\cH}f=F(x(t))\leq  \frac{\Gamma (t_0)}{t^{\alpha - 1}} +C_1\frac1{t^{\alpha - 1}} t^{\alpha -2r - 1} .
\end{equation}
We conclude that, as $t \to +\infty$
$$
f (x(t))- \inf_{\cH}f= \mathcal O\left( \frac{1}{t^{\alpha - 1}} + \frac{1}{t^{2r}} \right)
$$ 
As a consequence
\begin{itemize}
\item When $\alpha \geq 3$ we have
$$
f (x(t))- \inf_{\cH}f= \mathcal O\left(  \frac{1}{t^{2r}} \right)
$$
So we can get as close as possible to the optimal convergence rate $1/t^2$.

\item When $\alpha < 3$, by taking $ \frac{\alpha -1}{2}  <r<1$, we get
$$
f (x(t))- \inf_{\cH}f= \mathcal O\left(  \frac{1}{t^{\alpha - 1}} \right).
$$
%which is a little less good than the optimal convergence rate $1/t^{\frac{2\alpha}{3}}$.
\end{itemize}
%

%%%%%%%%%%%%%%%%%%%%%%%%
\noindent $(ii)$ \textit{ Strong convergence of $x(t)$ to the minimum norm solution}.

Let us go back to \eqref{eq:Eq-dif1}. We have
$$
\frac{d}{ds}(s^{\alpha - 1}x(s)) = ({\alpha - 1})s^{\alpha - 2}v(s).
$$
Integrating from $s_0$ to $t$, we get
\begin{eqnarray*}
x(t) &=& \left(\dfrac{s_0}{t}\right)^{\alpha - 1}x(s_0) + \frac{{\alpha - 1}}{t^{\alpha - 1}}\int_{s_0}^t \theta^{\alpha - 2}v(\theta)d\theta \\
&=&  \left(\dfrac{s_0}{t}\right)^{\alpha - 1}x(s_0) + \frac{{\alpha - 1}}{t^{\alpha - 1}}\int_{s_0}^t \theta^{\alpha - 2}(v(\theta)-x^*)d\theta + \left(\frac{{\alpha - 1}}{t^{\alpha - 1}}\int_{s_0}^t \theta^{\alpha - 2}d\theta\right) x^*\\
&=&  \left(\dfrac{s_0}{t}\right)^{\alpha - 1}x(s_0) + \frac{{\alpha - 1}}{t^{\alpha - 1}}\int_{s_0}^t \theta^{\alpha - 2}(v(\theta)-x^*)d\theta + \left(1-\left(\frac{s_0}{t}\right)^{\alpha - 1}\right) x^*\\
\end{eqnarray*}
Therefore,
\begin{eqnarray}\label{eq_cont-us2}
\|x(t) - x^*\| &\leq&   \left(\dfrac{s_0}{t}\right)^{\alpha - 1}|x(s_0) - x^*\| + \frac{{\alpha - 1}}{t^{\alpha - 1}}\int_{s_0}^t \theta^{\alpha - 2}\|v(\theta)-x^*\|d\theta.
\end{eqnarray}
By Theorem \ref{thm:model-a}, we have, as $\theta$ goes to $+\infty$,    strong convergence of $v(\theta)=z(\tau(\theta))$ to the minimum norm solution $x^*$. \\
Since $\lim_{s\rightarrow+\infty}\|v(\theta)-x^*\|=0$, then given $a > 0$, there exists $s_a>0$ sufficiently large so that $\|v(\theta)-x^*\|<a$, for $\theta \geq s_a$.

Then, for $t> s_a$, split the integral $\int_{s_0}^t \theta^{\alpha - 2}\|v(\theta)-x^*\|d\theta$ into two parts to obtain
\begin{eqnarray*}
\frac{\alpha - 1}{t^{\alpha - 1}}\int_{s_0}^t \theta^{\alpha - 2}\|v(\theta)-x^*\|d\theta &=& \frac{\alpha - 1}{t^{\alpha - 1}}\int_{s_0}^{s_a} \theta^{\alpha - 2}\|v(\theta)-x^*\|d\theta + \frac{\alpha - 1}{s^{\alpha - 1}}\int_{s_a}^t \theta^{\alpha - 2}\|v(\theta)-x^*\|d\theta\\
&\leq& \frac{\alpha - 1}{t^{\alpha - 1}}\int_{s_0}^{s_a} \theta^{\alpha - 2}\|v(\theta)-x^*\|d\theta + \frac{a(\alpha - 1)}{t^{\alpha - 1}}\int_{s_a}^t \theta^{\alpha - 2}d\theta\\
&=& a + \frac{1}{t^{\alpha - 1}} \left((\alpha - 1)\int_{s_0}^{s_a} \theta^{\alpha - 2}\|v(\theta)-x^*\|d\theta - as_a^{\alpha - 1}\right).
\end{eqnarray*}
Now let $t\rightarrow+\infty$ to deduce that
$$
\limsup_{t\rightarrow+\infty}\frac{\alpha - 1}{t^{\alpha - 1}}\int_{s_0}^t \theta^{\alpha - 2}\|v(\theta)-x^*\|d\theta \leq a.
$$
Since this is true for any $a > 0$, this insures 
$$
\lim_{t\rightarrow+\infty}\frac{\alpha - 1}{t^{\alpha - 1}}\int_{s_0}^t \theta^{\alpha - 2}\|v(\theta)-x^*\|d\theta=0.
$$
Going back to \eqref{eq_cont-us2}, we conclude for $\alpha - 1>0$ that $\lim_{t\rightarrow+\infty}\|x(t) - x^*\|=0$. This means that $x(t)$ strongly  converges to $x^*$ as $t\rightarrow +\infty$.

\medskip

$iii)$\textit{ Convergence of the gradients towards zero.} We have obtained in $ii)$ that the trajectory $x(\cdot)$ converges. Therefore, it is bounded. Let $L>0$ be the Lipschitz constant of   $\nabla f$ on a ball that contains the trajectory  $x(\cdot)$.   According to the convergence rate of values  and Lemma \ref{ext_descent_lemma} in the Appendix, we immediately obtain the following  convergence rate of the gradients towards zero
\vspace{-7pt}
\begin{equation*}
	\dfrac{1}{2L} \left\lVert \nabla f(x(t)) \right\rVert ^{2} \leq f(x(t)) -\min_{\cH} f =\mathcal O\left( \frac{1}{t^{\alpha - 1}} + \frac{1}{t^{2r}} \right).
\end{equation*}	
This completes the proof. \qed
\end{proof}

\section{Nonsmooth case}\label{sec:nonsmooth}
In this section, we adapt the  approach of \cite{ABotNguyen} to our context.
The pair of variables $(v,x)$ defined in section \ref{sect_second-order} satisfies the following differential system, which only involves first order derivatives in time and space
\begin{equation}\label{eq:fos2}
	\begin{cases}
		\dot{v} (t) +  \dfrac{t}{\alpha -1}\nabla f (v(t)) + \dfrac{t}{\alpha -1} \varepsilon \left(\dfrac{t^2}{2(\alpha -1)}\right) v(t)  & = 0
		\vspace{2mm}\\
		\dot x(t)+ \dfrac{\alpha -1}{t} (x(t) -v(t)) & = 0.
	\end{cases}
\end{equation}
\noindent This differential system naturally extends to the non-smooth case, replacing the gradient of $f$ by its subdifferential.
Given $f:\cH \to \rinf$ a convex, lower semicontinuous and proper function, this leads to consider the system of differential equation/inclusion
\begin{equation}\label{eq:fos2b}
	\begin{cases}
		\dot{v} (t) +  \dfrac{t}{\alpha -1}\partial f (v(t)) + \dfrac{t}{\alpha -1} \varepsilon \left(\dfrac{t^2}{2(\alpha -1)}\right) v(t)  & \ni  0
		\vspace{2mm}\\
		\dot x(t)+ \dfrac{\alpha -1}{t} (x(t) -v(t)) & = 0.
	\end{cases}
\end{equation}
Solving this system gives generalized solutions to the second-order differential inclusion
\begin{equation}\label{eq:fos2b2}
	\ddot x(t) +\frac{\alpha}{t}  \dot x(t) + \partial f\left[x(t)+\dfrac{t}{\alpha - 1}\dot x(t)\right] + \varepsilon \left(\dfrac{t^2}{2(\alpha - 1)}\right) \left[x(t)+\dfrac{t}{\alpha - 1}\dot x(ts)\right]\ni0,
\end{equation}
whose direct study raises several difficulties.
To extend the results  of the previous section to this nonsmooth case, we need to avoid the arguments using the Lipschitz continuity of $\nabla f$.
So, we are led to consider the Cauchy problem for \eqref{eq:fos2b2} with 
initial data $x_0 \in \dom f$ and $x_1=0$, that is with initial velocity equal to zero.
So doing we have $v(t_0)= x(t_0)=x_0$, which allows to interpret $x$ as an average of $v$.

Then, the  existence and uniqueness of a strong solution to the associated Cauchy problem relies on the equivalent formulation of \eqref{eq:fos2b} as a perturbation of the  generalized steepest descent dynamical system in the product space $\cH \times \cH$. Precisely, define $F:  \cH \times \cH \to \rinf$, which is given  for any $Z=(v,x) \in \cH \times \cH$ by
$$F(Z) = f(v),$$
and let  $G: [t_0, +\infty[ \times \cH \times \cH \rightarrow \cH \times \cH $ be the operator defined by
\begin{equation}
	G(t,Z)= \left(\dfrac{t}{\alpha -1} \varepsilon \left(\dfrac{t^2}{2(\alpha -1)}\right) v ,  \dfrac{\alpha -1}{t} (x -v) \right).
\end{equation}
Then \eqref{eq:fos2b} is written equivalently as
\begin{equation}
	\dot{Z}(t) + \frac{t}{\alpha -1}\partial F (Z(t)) +  G(t, Z(t))\ni 0.
\end{equation}
The initial condition becomes $Z(t_0)= (x_0,x_0)$ which belongs to 
$\dom F= \dom f \times \cH$.
According to the classical results concerning the Lipschitz perturbation of evolution equations governed by subdifferentials of convex functions,  see  \cite[Proposition 3.12]{Bre1}, we obtain the existence and uniqueness of a global strong solution of the Cauchy problem associated with \eqref{eq:fos2b}.
As a major advantage of the time scaling and averaging techniques, the  arguments used in the previous section still work in this more general nonsmooth situation. The rules of differential calculus are still valid for strong solutions, see \cite[chapter VIII.2]{Bre2}, and Jensen's inequality is still valid for a nonsmooth function $f$. Indeed Jensen's inequality is classical for a smooth convex function $f$. Its extension to the nonsmooth convex case can be obtained by first writing it for the Moreau-Yosida regularization $f_{\lambda}$ of $f$, then passing to the limit when $\lambda  \downarrow 0$. According to the monotone convergence of   $f_{\lambda}$ towards $f$, we can pass to the limit in the integral term thanks to the Beppo-Levy monotone convergence theorem.
We so obtain the following theorems corresponding 
to the cases $\varepsilon (t)=\frac{\delta}{t}$, and $\varepsilon (t)=\frac{1}{t^r}$.
\begin{theorem}\label{thm:model-inertial-b-nonsmooth}
Let $f:\cH \to \rinf$ be a convex, lower semicontinuous, and proper function such that $S=\argmin f \neq \emptyset$.
	Take    $\alpha >3$ and $\delta >1$, which gives $\alpha + 2\delta >5$.
	Let $x : [t_0, +\infty[ \to \mathcal{H}$ be a solution trajectory of
	\begin{equation}\label{2d-damped-id-p111-nonsmooth}
 \ddot x(t) +  \frac{\alpha + 2\delta}{t} \dot x(t) + \partial f\left[x(t)+\dfrac{t}{\alpha - 1}\dot x(t)\right] + \dfrac{2\delta (\alpha - 1)}{t^2} x(t)\ni 0
\end{equation}
which is taken in the sense of \eqref{eq:fos2b}, and which satisfies the initial conditions $x(t_{0}) \in \dom f$ and $\dot{x}(t_{0}) =0$.		
Then,  the following properties are satisfied. 	
\vspace{-2mm}
	\begin{eqnarray}
	&&(i) \,\,  f(x(t))-\min_{\cH} f= \mathcal O\left( \frac{1}{t^{2}}  \right) \mbox{ as } \; t \to +\infty;\label{contr:fx(t)22-nonsmooth}\\
	&& (ii) \,\, \mbox{ There is   strong convergence of } x(t)
	 \mbox{ to the minimum norm solution } x^*.\nonumber
	\end{eqnarray}
\end{theorem}
\begin{theorem}\label{thm:model-inertial-nonsmooth}
Let $f:\cH \to \rinf$ be a convex, lower semicontinuous, and proper function such that $S=\argmin f \neq \emptyset$.
	Take   $0<r< 1$ and $\alpha >1$.
	Let $x : [t_0, +\infty[ \to \mathcal{H}$ be a solution trajectory of
\begin{equation}\label{2d-damped-id-p1-t-nonsmooth}
\ddot x(t) + \left( \frac{\alpha}{t}  +\dfrac{2^r (\alpha - 1)^{r-1}}{t^{2r-1}}  \right)  \dot x(t) + \partial f\left(x(t)+\dfrac{t}{\alpha - 1}\dot x(t)\right) + 
\dfrac{2^r(\alpha - 1)^r}{t^{2r}} x(t)\ni 0
\end{equation}	
which is taken in the sense of \eqref{eq:fos2b}, and which satisfies the initial conditions $x(t_{0}) \in \dom f$ and $\dot{x}(t_{0}) =0$.		
Then,  the following properties are satisfied. 		
\vspace{-2mm}
	\begin{eqnarray}
	&&(i) \,\,  f(x(t))-\min_{\cH} f= \mathcal O\left( \frac{1}{t^{\alpha - 1}} + \frac{1}{t^{2r}} \right) \mbox{ as } \; t \to +\infty;\label{contr:fx(t)2-nonsmooth}\\
	&& (ii) \,\, \mbox{ There is   strong convergence of } x(t)
	 \mbox{ to the minimum norm solution } x^*.\nonumber
	\end{eqnarray}
\end{theorem}

\begin{remark}
As a consequence of $i)$, we have that $x(s)$ remains in the domain of $f$ for all $s\geq t_0$. This viability property strongly depends on the fact that the initial position belongs to the domain of $f$,  and that the initial velocity has been taken equal to zero.
% Note also that the  above theorems are valid in an infinite dimensional setting, which makes them applicable to nonlinear PDEs. 
\end{remark}

\section{Numerical illustrations}

The following simple examples illustrate the properties of the trajectories generated by the dynamics studied in the paper. They show that the trajectories verify both a rapid convergence of the values, the convergence towards the minimum norm solution, and a notable attenuation of the oscillations.
This highlights the additional properties obtained by the presence of the Tikhonov regularization term.
A detailed numerical study of these aspects should be the subject of further work.

\noindent We consider the  dynamical systems \eqref{edo001}, \eqref{edo001-0} and \eqref{1st-damped-id05} in $\mathbb R^2$:
\begin{eqnarray*}
\mbox{(AVD)}_{\alpha} &\hspace{1cm}& \ddot{x}(t) + \frac{\alpha}{t} \dot{x}(t) + \nabla f (x(t))=0,,\\
 \mbox{(AVD)}_{\alpha, \varepsilon} &&  \ddot{x}(t) + \frac{\alpha}{t} \dot{x}(t) + \Big(\nabla f  +\varepsilon(t)I\Big) x(t)=0, \vspace{2mm}\\
\hbox{(AVD-IH)}_{\alpha, \varepsilon} && \ddot x(t) +\frac{\alpha}{t}  \dot x(t) + \left(\nabla f  +\varepsilon \left(\dfrac{t^2}{2(\alpha - 1)}\right) I\right) \left(x(t)+\dfrac{t}{\alpha - 1}\dot x(t)\right)=0.
\end{eqnarray*}
Let us illustrate our results with the following two examples where the function $f$ is taken respectively strictly convex and convex with unbounded solution-set. 
 We choose $\alpha=3$, $\beta=1$ and $ \varepsilon(t)=t^{-r}$.
Our numerical tests were implemented in Scilab version 6.1 as an open source software.
\begin{example}\label{exemple1} 
Take $f_1  :  ]-1,+\infty[^{2} \to \R$ which is defined by $f_1(x)=\ln(1+e^{-(x+y)})+(x-y)^2+y.$
 The function $f_1$ is strictly convex with $x^*=(\frac18,-\frac18)$
the unique minimum. The trajectories corresponding to the three systems described above are shown in Fig. \ref{fig:trigs-c}.
\begin{figure} 
 \includegraphics[scale=0.35]{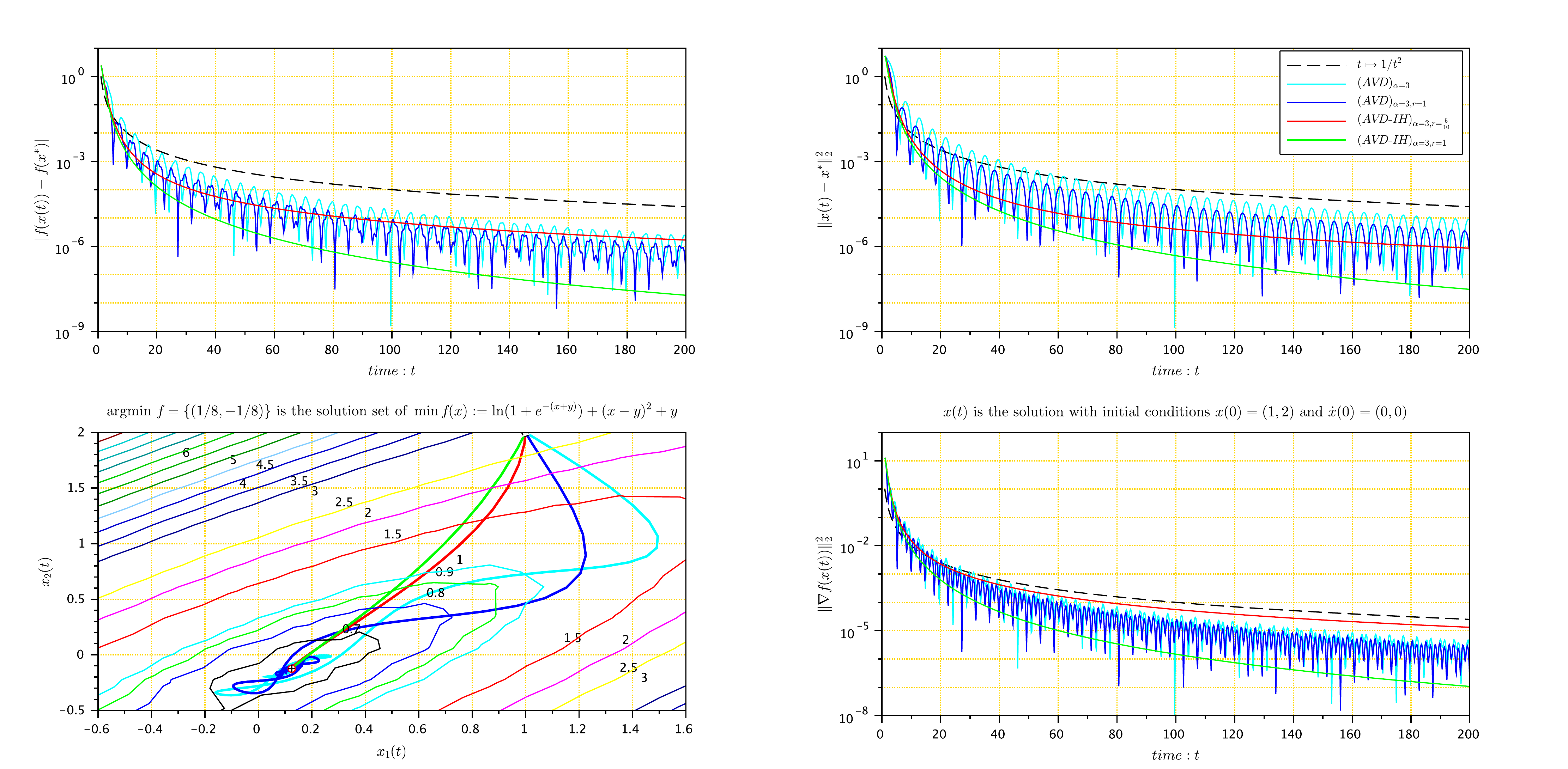}
  \caption{Convergence rates of
values $f_1(x(t))-f_1(x^*)$,  of trajectories $\|x(t)-x^*\|_2 $, and gradients  \; $\|\nabla f_1(x(t)\|_2$. Here $\hbox{argmin} f_1=\{\left( \frac18,-\frac18\right)\}$.}
 \label{fig:trigs-c}. 
\end{figure}
\end{example}
\begin{example}\label{exemple2} Consider the non-strictly convex  function $f_2  : \R^{2} \to \R$ defined by 
$f_2(x)=\frac12(x_1+x_2-1)^2 .$
The  set of solutions is $ \argmin f_2= \{(x_1,1-x_1): x_1\in\R\}$, and $x^*=(\frac12,\frac12)$ is the minimum norm solution. The trajectories corresponding to the systems  $\mbox{(AVD)}_{3}, \mbox{(AVD)}_{3, 1/t^2}$ and $\hbox{(AVD-IH)}_{3, 1/t^2}$ are represented in Fig. \ref{fig:trigs-convex}.

\begin{figure} 
  \includegraphics[scale=0.35]{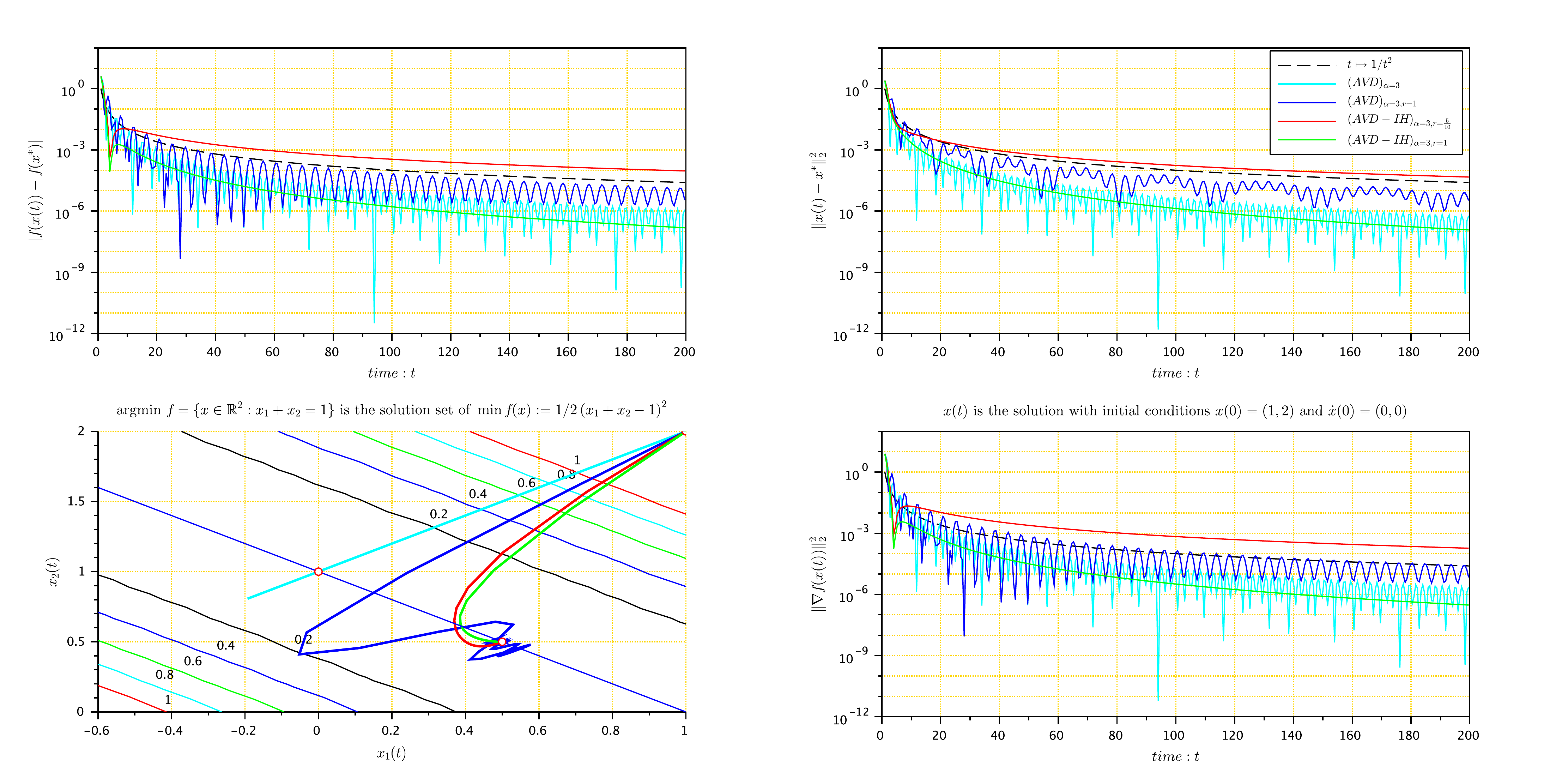}
  \caption{ Convergence rates of
values $f_2(x(t))-f_2(x^*)$, trajectories $\|x(t)-x^*\|_2^2$, and  gradients $\|\nabla f_2(x(t)\|_2^2$. Here $\hbox{argmin} f_1=\{(x,y)\in \mathbb R^2:x+y=1\}$. }
 \label{fig:trigs-convex} 
\end{figure}
\end{example}

\section{Conclusion, perspective}
The introduction of a Tikhonov regularization with vanishing coefficient in the optimization algorithms is beneficial in several respects.
For general convex optimization, instead of weak convergence of trajectories/iterations towards an optimal solution which depends on the initial state, it provides strong convergence towards  the minimum norm
solution. This is especially important for inverse problems, where one seeks a solution as close as possible to a desired state.
In this paper, we show that this can be achieved while preserving the fast convergence properties attached to the Nesterov accelerated gradient method.
Our approach is based on the Attouch-Bot-Nguyen scaling and time averaging technique, which proves to be flexible, and allows to address these issues in a unified way.
As a striking result, for the first time we obtain simultaneously the rapid convergence of the values $1/t^2$ and the strong  convergence towards the  minimum norm solution.
Let us mention some other interesting questions to examine further:

a) Our results open the way to the study of a new class of fast algorithms in convex optimization. Lyapunov's analysis of  continuous dynamics which support these algorithms will be a great help. It is probable that such algorithms share the good convergence properties of  continuous dynamics, according to the results obtained in \cite{ABotNguyen} concerning the time scaling and averaging method, and \cite{ABCR2} concerning the case without the Tikhonov regularization. 

b) An interesting and open question is whether a similar analysis can be developed using closed-loop Tikhonov regularization, i.e. the Tikhonov regularization coefficient is taken as a function of the current state of the system, see \cite{ABotCest} for a survey on these questions.

c) The Tikhonov regularization and the property of convergence to the minimum norm solution is a particular case of the general hierarchical principle which is attached to the viscosity method, see \cite{Att2}. In the context of numerical optimization, as an example, it is therefore natural to extend our study to the logarithmic barrier method for linear programming and to the selection of the analytical center.

\section{Appendix}
The following Lemma provides an extended  version of the classical gradient lemma which is valid for differentiable convex functions. The following version has been obtained in \cite[Lemma 1]{ACFR}, \cite{ACFR-Optimisation}.
We reproduce its proof for the convenience of the reader.

\begin{lemma}\label{ext_descent_lemma}
Let  $f: \cH \to \R$ be  a  convex function whose gradient is $L$-Lipschitz continuous. Let $s \in ]0,1/L]$. Then for all $(x,y) \in \cH^2$, we have
\begin{equation}\label{eq:extdesclem}
f(y - s \nabla f (y)) \leq f (x) + \left\langle  \nabla f (y), y-x \right\rangle -\frac{s}{2} \|  \nabla f (y) \|^2 -\frac{s}{2} \| \nabla f (x)- \nabla f (y) \|^2 .
\end{equation}
In particular, when $\argmin f \neq \emptyset$, we obtain that for any $x\in \cH$
\begin{equation}\label{eq:extdesclemb}
f(x) \geq \min_{\cH} f  +\frac{1}{2L} \| \nabla f (x)\|^2 .
\end{equation}
\end{lemma}

\begin{proof}
Denote $y^+=y - s \nabla f (y)$. By the standard descent lemma applied to $y^+$ and $y$, and since $sL \leq 1$ we have
\begin{equation}\label{eq:descfm2}
f(y^+) \leq f(y) - \frac{s}{2}\pa{2-Ls} \| \nabla f (y) \|^2 \leq f(y) - \frac{s}{2} \|  \nabla f (y) \|^2.
\end{equation}
We now argue by duality between strong convexity and Lipschitz continuity of the gradient of a convex function. Indeed, using Fenchel identity, we have
\[
f(y) = \dotp{\nabla f(y)}{y} - f^*(\nabla f(y)) .
\]
$L$-Lipschitz continuity of the gradient of $f$ is equivalent to $1/L$-strong convexity of its conjugate $f^*$. This together with the fact that $(\nabla f)^{-1}=\partial f^*$ gives for all $(x,y) \in \cH^2$,
\[
f^*(\nabla f(y)) \geq  f^*(\nabla f(x)) + \dotp{x}{\nabla f(y)-\nabla f(x)} + \frac{1}{2L}\norm{\nabla f(x)-\nabla f(y)}^2 .
\]
Inserting this inequality into the Fenchel identity above yields
\begin{align*}
f(y) 
&\leq - f^*(\nabla f(x)) + \dotp{\nabla f(y)}{y} - \dotp{x}{\nabla f(y)-\nabla f(x)} - \frac{1}{2L}\norm{\nabla f(x)-\nabla f(y)}^2 \\
&= - f^*(\nabla f(x)) + \dotp{x}{\nabla f(x)} + \dotp{\nabla f(y)}{y-x} - \frac{1}{2L}\norm{\nabla f(x)-\nabla f(y)}^2 \\
&= f(x) + \dotp{\nabla f(y)}{y-x} - \frac{1}{2L}\norm{\nabla f(x)-\nabla f(y)}^2  \\
&\leq f(x) + \dotp{\nabla f(y)}{y-x} - \frac{s}{2}\norm{\nabla f(x)-\nabla f(y)}^2 .
\end{align*}
Inserting the last bound into \eqref{eq:descfm2} completes the proof.
\end{proof}

\end{document}